\documentclass[a4paper,11pt]{article}
\usepackage{mathrsfs}
\usepackage{amsfonts}
\usepackage{amsfonts}
\usepackage{amssymb}
\usepackage{mathrsfs}
\usepackage{indentfirst,latexsym,bm,amsmath,amssymb,amsthm}
\usepackage{xcolor}
\usepackage[
                   bookmarksnumbered=true,
                   bookmarksopen=true, 
                   pdfauthor=kmc,
                   pdfcreator={LaTex with hyperref package + WinEdt},
                   pdftitle=trudge,
                   colorlinks,%
                   citecolor=blue,
                   linkcolor=blue,%
                   urlcolor=blue,
                   hyperindex,%
                   plainpages=false,%
                   pdfstartview=FitH,
                   linktocpage=true,
                   dvipdfm%
                   ]{hyperref}
\usepackage{indentfirst,latexsym,bm,amsmath,amssymb,amsthm}
\usepackage[dvips]{graphicx}

\makeatletter\@addtoreset{equation}{section} \makeatother
\newtheorem{thm}{Theorem}[section]
\newtheorem{Lemma}{Lemma}[section]
\newtheorem{proposition}{Proposition}[section]

\newtheorem{rem}{Remark}[section]

 

\title{ Space-time $L^2$ estimates, regularity and almost global existence for elastic waves}

\author{Kunio Hidano\thanks{Department of Mathematics, Faculty of Education, Mie University,
1577 Kurima-machiya-cho, Tsu, Mie 514-8507, Japan.
{ E-mail address: hidano@edu.mie-u.ac.jp}} ~~~~  Dongbing Zha\thanks{ Corresponding author. Department of Applied Mathematics, Donghua University, Shanghai 201620, PR China.{ E-mail address: ZhaDongbing@163.com}     }}

\begin{document}

\maketitle

\begin{abstract}
In this paper, we first establish a kind of weighted space-time $L^2$ estimate, which belongs to Keel-Smith-Sogge type estimates, for perturbed linear elastic wave equations. This estimate refines the corresponding one established by the second author [\href{http://dx.doi.org/10.1016/j.jde.2017.04.014}{J. Differential Equations 263
  (2017), 1947--1965}] and is proved by combining the methods in the former paper, the first author, Wang and Yokoyama's paper
 [\href{http://projecteuclid.org/euclid.ade/1355703087}{Adv. Differential Equations 17 (2012), 267--306}], and some new ingredients. Then together with some weighted Sobolev inequalities, this estimate is used to show a refined version of almost global existence of classical solutions for nonlinear elastic waves with small initial data. Compared with former almost global existence results for nonlinear elastic waves due to John [\href{http://dx.doi.org/10.1002/cpa.3160410507}{Comm. Pure Appl. Math. 41 (1988) 615--666}], Klaierman-Sideris [\href{http://dx.doi.org/10.1002/(SICI)1097-0312(199603)49:3<307::AID-CPA4>3.0.CO;2-H}{Comm. Pure Appl. Math. 49 (1996) 307--321}], the main innovation of our one is that it considerably improves the amount of regularity
of initial data, i.e., the Sobolev regularity of initial data is assumed to be the smallest among all the admissible Sobolev spaces of integer order in the standard local existence theory. Finally, in the radially symmetric case,
we establish the almost global existence of a low regularity solution for every small
initial data in $H^3\times H^2$.  \\
\emph{keywords}: Elastic waves; Keel-Smith-Sogge type estimates; regularity; almost global existence 
\end{abstract}

\pagestyle{plain} \pagenumbering{arabic}

\section{Introduction and main results}
As a significant kind of space-time $L^2$ estimates for wave equations, the Keel-Smith-Sogge~(KSS for short) type estimate, which was first introduced in Keel, Smith and Sogge \cite{MR1945285},  is largely recognized as a key to giving a simple proof of long time existence of small solutions to nonlinear wave equations
(see, e.g., Section 2.3 of \cite{MR2455195}). In particular,  since Rodnianski \cite{MR2128434} discovered the
multiplier method for the proof of this kind of estimates, for studying quasilinear wave equations, the KSS type estimate for linear wave operators
with variable coefficients has drawn many attentions (see \cite{MR2217314}, \cite{MR2299569}, \cite{MR2919103} etc.). For perturbed linear elastic wave equations, the second author \cite{MR3650329} proved a KSS type estimate where the weight is a negative power of $\langle x\rangle$. A full history of KSS type estimates can be found in \cite{MR3650329}. \par
In this paper, we first revisit \cite{MR3650329} to present a refinement in that the KSS estimate thereby obtained in \cite{MR3650329} remains to hold
even for the weight of the form of a negative power of $\langle x\rangle^{2\delta}|x|^{1-2\delta}$
for small $\delta>0$.  To achieve this goal, we will use the framework in \cite{MR3650329}, the multiplier method of Rodnianski and
the multiplier (suggested by Metcalfe) used in Hidano, Wang and Yokoyama \cite{MR2919103}, where such weight has been used (see also \cite{MR2262091}) in the context of wave equations with low-regularity radial data, and some new ingredients associated with the distinguishing features of elastic waves.  We can also find a significant difference from the earlier version of \cite{MR3650329}
in that the present version of the KSS estimate leads to not only a simple proof, but also considerable improvement of the earlier results due to
John \cite{John88}, Klainerman-Sideris \cite{Klainerman96}, and the second author \cite{MR3650329} concerning almost global existence of small solutions to 3-D nonlinear elastic wave equations.
Recall that the standard local existence theorem
was established in the Sobolev space $H^{s+1}\times H^s$ with $s>5/2$ (see Hughes, Kato and Marsden \cite{MR0420024}), which has motivated us to obtain
almost global existence result in $H^4\times H^3$, the lowest integer order Sobolev space.
Our theorem thereby proved requires the smallness of data with respect to some weighted $H^4\times H^3$ norm,
while, say, the results in \cite{Klainerman96} and \cite{MR3650329} required it with respect to some weighted $H^7\times H^6$ norm and weighted $H^{10}\times H^9$ norm, respectively.
Compared with \cite{Klainerman96}, there exists another advantage in the proof of almost global existence
on the basis of the KSS estimate: it never uses the scaling operator $S=t\partial_t+x\cdot\nabla$.
We benefit from this advantage, especially when studying radially symmetric solutions
(see page 6 below concerning the definition of radial symmetry of ${\mathbb R}^3$-valued functions, see also Remark 4.1 below).
Owing to ${\widetilde\Omega}_{ij}u\equiv 0$ for radial functions (see \eqref{below222} below for the definition of
the generators of simultaneous rotations ${\widetilde\Omega}_{ij}$),
we have almost global existence for small $H^4\times H^3$ radial data.
Note that no additional decay is required. Actually, by our analysis we can go further below
the Sobolev space $H^{s+1}\times H^s$ $(s>5/2)$
where, as mentioned above,
Hughes, Kato, and Marsden established the local existence theorem. That is, we prove almost global existence for small
$H^3\times H^2$ radial data.
\par
 Before the further formulation of our results, let us first review the history on the topics considered here.  The long time existence of classical solutions for nonlinear elastic waves started from Fritz John's pioneering work on elastodynamics (see~Klainerman \cite{Klainerman98}). John \cite{John84} showed that local classical solutions in general develop singularities for radial and small initial data, and they almost globally exist \cite{John88}. See also simplified proofs in \cite{Klainerman96}, \cite{MR3650329} and a lower bound estimate in \cite{MR3014806}.
In order to ensure the global existence of classical solutions with small initial data, some structural condition on the nonlinearity which is called the null condition is necessary. We refer the reader to Agemi \cite{Agemi00} and~Sideris \cite{Sideris00} (see also a previous result in Sideris \cite{Sideris96}).
The exterior domain analogues of John's almost global existence result and Agemi and~Sideris's global existence result were obtained in \cite{MR2249996} and \cite{MR2344575}, respectively. We note that all the above works are in the framework of classical solutions and very high regularity on the initial data is required.\par
Now we give a brief review for the low regularity well-posedness for the Cauchy problem of nonlinear wave equations.
Studies on local low regularity well-posedness for the Cauchy problem of nonlinear
wave equations started from Klainerman-Machedon's pioneering work \cite{MR1231427} in the semilinear case. For general quasilinear equations, we refer the reader to the sharpest result \cite{MR2178963}, \cite{MR3656947} and the references therein. On the other hand, Lindblad constructed
some counterexamples to show sharp results of ill-posedness \cite{MR1375301}, \cite{MR1666844}.
It should be mentioned that
improvement of regularity in the presence of
radial symmetry was first observed in \cite{MR1231427}
for semilinear equations such as
$\Box u=c_1(\partial_t u)^2+c_2|\nabla u|^2$.
(See also page 176 of \cite{MR1211729},
Section 5 of \cite{MR2108356}, and
 \cite{MR2262091},
 \cite{MR2971205}, and
\cite{Hidano17} for related results.)
We also mention that
it was also observed for the wave equation with
a power-type nonlinear term $\Box u=F(u)$,
first by Lindbald and Sogge \cite{MR1335386}, and then by
some authors (see  \cite{MR2569550},
 \cite{MR2911108},  \cite{MR3286047}), and quasilinear wave equations of the form
$\partial_t^2u-a^2(u)\Delta u=c_1(\partial_t u)^2+c_2|\nabla u|^2$, by \cite{MR2919103} and  \cite{MR2482537}.
In particular,
the almost global existence of low regularity radially symmetric solutions with small initial data was showed in \cite{MR2262091} and \cite{MR2919103} for the semilinear and the quasilinear case in 3-D, respectively. A global
regularity result can be found in \cite{MR2482537}, which is the first result for global existence of low regularity
solutions to 3-D quasilinear wave equations. We also note that the global existence for some 4-D quasilinear wave equations with low regularity was obtained in \cite{Liu} recently.  By the analysis of the present paper,
improvement of regularity for almost global solutions can be successfully observed for nonlinear
elastic wave equations under the assumption of radial symmetry.\par
The outline of this paper is as follows.
The remainder of this section will be
devoted to the exact description of three main results in this paper, i.e., the space-time $L^2$ estimates for perturbed linear elastic wave equations (Theorem \ref{thmKSS00}), refined version of almost global existence for nonlinear elastic wave equations with small weighted $H^4\times H^3$ norm (Theorem \ref{mainthm}) and low regularity almost global existence in the radially symmetric case for nonlinear elastic wave equations with small $H^3\times H^2$ norm (Theorem \ref{mainthm33}). In the next section, we will prove Theorem \ref{thmKSS00}. The proof of Theorem \ref{mainthm} will be given in Section \ref{sec3} and the proof of Theorem \ref{mainthm33} will be given in Section \ref{sec3hhh}. They are all based on
Theorem \ref{thmKSS00} and some weighted Sobolev inequalities.

\subsection{Space-time $L^2$ estimates for perturbed linear elastic wave equations}
Consider the following~operator corresponding to
perturbed linear elastic waves
\begin{align}\label{elasticwave345}
L_{h}u=L u+Hu,
\end{align}
where the linear elastic wave operator
\begin{align}
Lu&=\partial_t^2u-c_2^2\Delta u-(c_1^2-c_2^2)\nabla\nabla\cdot u,
\end{align}
and the perturbed term\footnote{We use the summation convention over repeated indices.}
\begin{align}\label{Hdefin}
(Hu)^{i}&=\partial_{l}(h^{ij}_{lm}(t,x)\partial_mu^{j}),~~i=1,2,3.
\end{align}
Here $u(t,x)=(u^{1}(t,x),u^{2}(t,x),u^{3}(t,x))$ denotes the displacement vector from the
reference configuration,
and the material constants $c_1$ (pressure wave speed) and $c_2$ (shear wave speed) satisfy $0<c_2<c_1$.\par
\par
The first main result in this paper is the following
\begin{thm}\label{thmKSS00}
Assume that~$h=(h^{ij}_{lm})$ satisfies the following symmetric condition
\begin{align}\label{duichengxing678}
h^{ij}_{lm}=h^{ji}_{ml},
\end{align}
and the smallness condition
\begin{align}\label{smallness}
|h|=\sum_{i,j,l,m=1}^{3}|h^{ij}_{lm}|\ll 1.
\end{align}
Denote
 the strip by $S_T=[0,T]\times \mathbb{R}^3$.
 Suppose that $u\in C^{\infty}([0,T];C_{0}^{\infty}(\mathbb{R}^3;\mathbb{R}^3))$.  Then we have
 \begin{align}\label{KSS0}
&\sup_{0\leq t\leq T}\|\partial u\|_{L^2(\mathbb{R}^{3})}+ (\log(2+T))^{-1/2}\|\langle r\rangle^{-1/2}\partial u\|_{L^2_{t,x}(S_{T})}\nonumber\\
&~~~~~~~~~~~~~~~~~~~~~+
 (\log(2+T))^{-1/2}\|\langle r\rangle^{-3/2} u\|_{L^2_{t,x}(S_{T})}
 \nonumber\\
&\leq C\|\partial u(0,\cdot)\|_{L^2(\mathbb{R}^{3})}+
C\|L_hu\|_{L^1_{t}L^2_{x}(S_{T})}&\nonumber\\
&~~~~~~~~~~~~~~~~~~~~~~~~~+C\||\partial h||\nabla u|\|_{L^1_{t}L^2_{x}(S_{T})}+C\|{\langle r\rangle}^{-1}{|h|}|\nabla u|\|_{L^1_{t}L^2_{x}(S_{T})}&
\end{align}
and for any $0<\delta<1/2$,
\begin{align}\label{KSS111}
&\sup_{0\leq t\leq T}\|\partial u\|_{L^2(\mathbb{R}^{3})}+ (\log(2+T))^{-1/2}\|\langle r\rangle^{-\delta}r^{-1/2+\delta}\partial u\|_{L^2_{t,x}(S_{T})}\nonumber\\
&~~~~~~~~~~~~~~~~~~~~~+
(\log(2+T))^{-1/2}\|\langle r\rangle^{-\delta}r^{-3/2+\delta} u\|_{L^2_{t,x}(S_{T})}\nonumber\\
&\leq C\|\partial u(0,\cdot)\|_{L^2(\mathbb{R}^{3})}+
C\|L_hu\|_{L^1_{t}L^2_{x}(S_{T})}+C\||\partial h||\nabla u|\|_{L^1_{t}L^2_{x}(S_{T})}\nonumber\\
&~~+C\|{\langle r\rangle}^{-1}{|h|}|\nabla u|\|_{L^1_{t}L^2_{x}(S_{T})}
+C(\log(2+T))^{1/2}\|\langle r\rangle^{-\delta}{ r}^{-{1}/{2}+\delta}{|h|}|\nabla u|\|_{L^2_{t,x}(S_{T})},
\end{align}
where $r=|x|$ and $\langle \cdot\rangle=(1+|\cdot|^2)^{1/2}$, here and in what follows $C$ denotes a positive constant.
\end{thm}
\subsection{Refined almost global existence for nonlinear elastic wave equations}
Consider the following Cauchy problem for homogeneous, isotropic and hyperelastic waves:
 \begin{align}\label{Cauchy}
 \begin{cases}
 \partial_t^2u-c_2^2\Delta u-(c_1^2-c_2^2)\nabla\nabla\cdot u=N( u, u),\\
 t=0:u= u_0,u_t= u_1.
 \end{cases}
 \end{align}
 Here the nonlinear term $N$ is taken as
 \begin{align}\label{cutifof}
 N( u, v)^{i}=\partial_{l}(g^{ijk}_{lmn}\partial_m u^{j}\partial_nv^{k}),
 \end{align}
where the coefficients
are constants and are symmetric with respect to pairs of indices
 \begin{align}\label{sym1}
 g^{ijk}_{lmn}=g^{jik}_{mln}=g^{kji}_{nml},
 \end{align}
 and $g=(g^{ijk}_{lmn})$ is a six order isotropic tensor.
 \par
The angular momentum operators (generator of the spatial rotation) are the vector fields
\begin{align}
\Omega=(\Omega_{ij}: 1\leq i<j\leq 3),
\end{align}
where
\begin{align}
\Omega_{ij}=x_i\partial_j-x_j\partial_i.
\end{align}
Denote the generators of
simultaneous rotations by $\widetilde{\Omega}=(\widetilde{\Omega}_{ij}: 1\leq i<j\leq 3)$, where
\begin{align}\label{below222}
\widetilde{\Omega}_{ij}=\Omega_{ij} I+U_{ij},
\end{align}
and $U_{ij}=e_i\otimes e_j-e_j\otimes e_i$, $\{e_i\}_{i=1}^{3}$ is the standard basis
on $\mathbb{R}^3$.
Denote the collection of vector fields by $Z=(\nabla,\widetilde{\Omega})$, which is time-independent. It is easy to verify the following commutation relationship
\begin{align}\label{compoi}
\big[\widetilde{\Omega}, \nabla\big]=\nabla.
\end{align}
 By $Z^{a}$, $a=(a_1,\dots,a_k)$, we denote an ordered product of $k=|a|$ vector fields $Z_{a_1}\cdots Z_{a_k}$\footnote{{{Note that the notation used for $Z^a$ differs from the standard multi-index notation.}}}.
\par
 Define the time-independent spaces
\begin{align}
H^{k}_{Z}=\{f\in L^2(\mathbb{R}^3;\mathbb{R}^3): Z^{a}f\in L^2(\mathbb{R}^3;\mathbb{R}^3),|a|\leq k\}
\end{align}
with the norm
\begin{align}
\|f\|_{Z,k}=\Big(\sum_{|a|\leq k}\|Z^{a}f\|^2_{L^2(\mathbb{R}^3)}\Big)^{1/2}.
\end{align}
\par
The solution will be constructed in the space $X^{k}_{Z}(T)$ obtained
by closing the set\\ $C^{\infty}([0,T];C_{0}^{\infty}(\mathbb{R}^3;\mathbb{R}^3))$
 in the norm $\|u\|_{{X}^{k}_{Z}(T)}$, where
  \begin{align}
\|u\|_{{X}^{k}_{Z}(T)}=&\sum_{|a|\leq k-1}\|\partial Z^{a}u\|_{L^{\infty}_{t}L^2_{x}(S_T)}+ (\log(2+T))^{-1/2}\sum_{|a|\leq k-1}\|\langle r\rangle^{-\delta}r^{-1/2+\delta}\partial Z^{a}u\|_{L^2_{t,x}(S_{T})}\nonumber\\
&+(\log(2+T))^{-1/2}\sum_{|a|\leq k-1}\|\langle r\rangle^{-\delta}r^{-3/2+\delta} Z^{a}u\|_{L^2_{t,x}(S_{T})},
\end{align}
and $0<\delta\leq 1/4$ is fixed.
 \par
 The second main result in this paper is the following
 \begin{thm}\label{mainthm}
There exist constants $\kappa, \varepsilon_0>0$ such that for any given $\varepsilon$ with $0<\varepsilon\leq \varepsilon_{0}$, if
\begin{align}\label{label345}
\|\nabla u_0\|_{Z,3}+\|u_1\|_{Z,3}\leq \varepsilon,
\end{align}
then Cauchy problem \eqref{Cauchy} admits a unique classical solution $u\in {X}^{4}_{Z}(T_{\varepsilon})$ with
\begin{align}
T_{\varepsilon}=\exp({\kappa}/{\varepsilon}).
\end{align}
\end{thm}
\begin{rem}
Using {\rm (3.2)} in Klainerman and Sideris {\rm \cite{Klainerman96}}, following the argument in Sideris and Tu {\rm \cite{Sideris01}} and adapting the argument {\rm (3.11)--(3.13)} and {\rm (4.18)--(4.26)} in Hidano {\rm \cite{Hidano}}, we are able to prove almost global existence under the slightly more restrictive condition
\begin{align}
\|\nabla u_0\|_{Z,3}+\|u_1\|_{Z,3}+\|r \partial_r\nabla u_0\|_{Z,2}+\|r \partial_ru_1\|_{Z,2}\leq \varepsilon.
\end{align}
The last two norms on the left-hand side above come from the use of the scaling operator $S=t\partial_t+r\partial_r$. It is well known that the method in Klainerman and Sideris {\rm \cite{Klainerman96}} and Sideris {\rm \cite{Sideris00}} works for the proof of global existence under the null condition.
\end{rem}
\begin{rem}
For radially symmetric initial data\footnote{A vector function $v$ is called radial if it has the form $v(x)=x\phi(r)~ (r=|x|)$, where $\phi$ is a scalar radial function. We refer the reader to Definition 4.4 and Lemma 4.5 of \cite{MR1774100}.}$u_0,u_1$, noting that $\widetilde{\Omega}u_0=\widetilde{\Omega}u_1=0$,
 we easily see the condition \eqref{label345}
is satisfied whenever the norm
\begin{align}\label{label345ddd}
\|\nabla u_0\|_{H^3}+\|u_1\|_{H^3}
\end{align}
is small enough. Note that we assume no additional decay. The result of almost global existence for nonlinear elastic wave equations is new when smallness
is required of only such Sobolev norm of radial data. In fact, we will even lower the regularity requirement on initial data by one derivative in the next subsection.
\end{rem}
\subsection{Low regularity almost global existence in the radially symmetric case}
Noting the nonlinear elastic wave equation is invariant under simultaneous rotations (see page 860 of \cite{Sideris00}), we can seek radially symmetric solution for the Cauchy problem \eqref{Cauchy} with radially symmetric initial data. \par
Consider the space $X^{k}(T)$, which is obtained
by closing the set\\ $C^{\infty}([0,T];C_{0}^{\infty}(\mathbb{R}^3;\mathbb{R}^3))$
 in the norm $\|u\|_{{X}^{k}(T)}$, where
\begin{align}
\|u\|_{X^{k}(T)}=&\sum_{|a|\leq k-1}\|\partial \nabla^{a}u\|_{L^{\infty}_{t}L^2_{x}(S_T)}+ (\log(2+T))^{-1/2}\sum_{|a|\leq k-1}\|\langle r\rangle^{-\delta}r^{-1/2+\delta}\partial \nabla^{a}u\|_{L^2_{t,x}(S_{T})}\nonumber\\
&+(\log(2+T))^{-1/2}\sum_{|a|\leq k-1}\|\langle r\rangle^{-\delta}r^{-3/2+\delta} \nabla^{a}u\|_{L^2_{t,x}(S_{T})},
\end{align}
and $0<\delta\leq 1/4$ is fixed.
The solution will be constructed in the space
\begin{align}
X_{\text{rad}}^{k}(T)=\big\{ u ~\text{is radially symmetric}, \|u\|_{X^{k}(T)}<+\infty\big\}.
\end{align}
\par
 The third main result in this paper is the following
 \begin{thm}\label{mainthm33}
There exist constants $\kappa, \varepsilon_0>0$ such that for any given $\varepsilon$ with $0<\varepsilon\leq \varepsilon_{0}$, if the initial data is radially symmetric and satisfies
\begin{align}\label{label345newe}
\|\nabla u_0\|_{H^2}+\|u_1\|_{H^2}\leq \varepsilon,
\end{align}
then Cauchy problem \eqref{Cauchy} admits a unique solution $u\in X_{\rm rad}^{3}(T_{\varepsilon})$ with
\begin{align}
T_{\varepsilon}=\exp({\kappa}/{\varepsilon}).
\end{align}
\end{thm}

\section{Proof of Theorem \ref{thmKSS00}}
The strategy of our proof of Theorem \ref{thmKSS00} is as follows. As in \cite{MR3650329},
we first use the Hodge decomposition to decompose the solution into its curl-free and
divergence-free components, then employ the multiplier method of Rodnianski
(see Appendix of \cite{MR2128434}) which works for the wave equation with
variable coefficients. Here we use the multiplier in \cite{MR2919103} which was suggested by Metcalfe.
Inevitably, the presence of the inverse of the Laplacian in front of
the perturbation terms then becomes a major obstacle to success.
Outside of the integration-by-part argument,
the boundedness of the Riesz operators
on the Lebesgue space $L^2$ with the Muckenhoupt $A_2$ weight,
together with the divergence form of the perturbation terms,
plays a crucial role in handling such terms.  Note that on the right-hand side of \eqref{KSS0} and \eqref{KSS111}, we have to take $L^2$ norm in spatial variables. Particularly in \eqref{KSS111}, for the consideration of nonlinear applications in the following two sections, we have to avoid some terms which are too singular near $r=0$. See \eqref{shangj}. This is why a space-time $L^2$
norm is taken in the last term on the right-hand side of \eqref{KSS111}, which is also the main distinction between our KSS estimate \eqref{KSS111} for elastic waves and the corresponding one for wave equations such as (2.5) in \cite{MR2919103}.

First we introduce the Hodge decomposition in $\mathbb{R}^{3}$. For function~$u\in C^{\infty}(\mathbb{R}^{3};\mathbb{R}^{3})$, denote the curl of $u$ by $\nabla\times u$.
The following lemma is the Hodge decomposition, which projects any vector field onto its curl-free and
divergence-free components. It is widely
known and the proof can be found in \cite{MR1218879}.
\begin{Lemma}\label{hodgefen}
For any function~$u\in C^{\infty}(\mathbb{R}^{3};\mathbb{R}^{3})$ with sufficient decay at infinity, we have
\begin{align}\label{hodge1}
u=u_{cf}+u_{df},
\end{align}
where
\begin{align}\label{hodge2}
u_{cf}=\nabla\nabla\cdot \Delta^{-1}u,~~     u_{df}=-\nabla\times\nabla\times \Delta^{-1}u.
\end{align}
And it holds that
\begin{align}\label{hodge3}
\nabla\times u_{cf}=0,~~  \nabla\cdot u_{df}=0,
\end{align}
\begin{align}\label{hodge456}
\|u\|^2_{L^2(\mathbb{R}^3)}=\|u_{cf}\|^2_{L^2(\mathbb{R}^3)}+\|u_{df}\|^2_{L^2(\mathbb{R}^3)},
\end{align}
\begin{align}\label{hodge456789}
\|\nabla u\|^2_{L^2(\mathbb{R}^3)}=\|\nabla u_{cf}\|^2_{L^2(\mathbb{R}^3)}+\|\nabla u_{df}\|^2_{L^2(\mathbb{R}^3)}.
\end{align}
\end{Lemma}
In order to treat the effect of the inverse of the Laplacian,
we will also use the following
\begin{Lemma}\label{yaoyaohj}
Let $0<\delta<{1}/{2}$. For the Riesz transformation $R_i=\frac{\partial_i}{\sqrt{-\Delta}}$, we have
\begin{align}\label{holds222}
\|\langle x\rangle^{-\delta} | x|^{-1/2+\delta}R_if\|_{L^2(\mathbb{R}^{3})}\leq C\|\langle x\rangle^{-\delta} | x|^{-1/2+\delta}f\|_{L^2(\mathbb{R}^{3})}.
\end{align}
\end{Lemma}
\begin{proof}
When $0<\delta<{1}/{2}$, the weight $\langle x\rangle^{-2\delta}| x|^{-1+2\delta}$ belongs to the Muckenhoupt class $A_1$ (see Lemma 2.5 of \cite{Hidano17}). Since $A_1\subset A_2,$ it also belongs to the Muckenhoupt class $A_2$. Thanks to this fact, we enjoy the $L^2(\mathbb{R}^{3},\langle x\rangle^{-2\delta}| x|^{-1+2\delta}dx)$  boundedness of
   the Riesz transformation,  i.e., \eqref{holds222} holds.
 See page 205 of Stein \cite{MR1232192}.
\end{proof}
Now we prove Theorem \ref{thmKSS00}.
We note that \eqref{KSS0} has been proved in \cite{MR3650329}. Actually, from the proof of \eqref{KSS0}, it is easy to find that it also holds that
\begin{align}\label{KSS00}
& (\log(2+T))^{-1/2}\|\langle r\rangle^{-1/2}\partial u_{cf}\|_{L^2_{t,x}(S_{T})}+(\log(2+T))^{-1/2}\|\langle r\rangle^{-3/2} u_{cf}\|_{L^2_{t,x}(S_{T})}\nonumber\\
&+ (\log(2+T))^{-1/2}\|\langle r\rangle^{-1/2}\partial u_{df}\|_{L^2_{t,x}(S_{T})}+(\log(2+T))^{-1/2}\|\langle r\rangle^{-3/2} u_{df}\|_{L^2_{t,x}(S_{T})}\nonumber\\
&\leq C\|\partial u(0,\cdot)\|_{L^2(\mathbb{R}^{3})}+
C\|L_hu\|_{L^1_{t}L^2_{x}(S_{T})}+C\||\partial h||\nabla u|\|_{L^1_{t}L^2_{x}(S_{T})}+C\|{\langle r\rangle}^{-1}{|h|}|\nabla u|\|_{L^1_{t}L^2_{x}(S_{T})},
\end{align}
which will be useful in the proof of \eqref{KSS111}.\par
Now we will give the proof of \eqref{KSS111}.
For the following  classical energy estimate
\begin{align}\label{basicene}
&\sup_{0\leq t\leq T}\|\partial u\|^2_{L^2(\mathbb{R}^{3})}\leq C\|\partial u(0,\cdot)\|^2_{L^2(\mathbb{R}^{3})}
+C\||\partial_th||\nabla u|\|^2_{L^1_tL^2_x(S_T)}+C\|L_hu\|^2_{L^1_{t}L^2_{x}(S_T)},
\end{align}
we omit its proof because it has been given in the literature (see, e.g., page 1951 of \cite{MR3650329}).
\par
Now we give the proof of the weighted
 space-time $L^2$ estimates.
By~Lemma \ref{hodgefen}, we have
\begin{align}\label{hodge4}
u=u_{cf}+u_{df}, ~~Hu=(Hu)_{cf}+(Hu)_{df}.
\end{align}
Denote~the wave operator by $\Box_c=\partial_t^2-c^2\Delta$. It can be verified that
\begin{align}\label{hodge5}
(L_hu)_{cf}=\Box_{c_1}u_{cf}+(Hu)_{cf}, ~~(L_hu)_{df}=\Box_{c_2}u_{df}+(Hu)_{df}.
\end{align}
Denote the multiplier vector field
\begin{align}\label{multiplier}
M=f(r)\partial_r+\frac{1}{r}f(r),
\end{align}
where $f(r)$ will be determined later.
It follows from \eqref{hodge4} and \eqref{hodge5} that
\begin{align}
&\langle Mu_{cf}, \Box_{c_1}u_{cf}\rangle+\langle Mu_{df}, \Box_{c_2}u_{df}\rangle\nonumber\\
&=-\langle Mu_{cf},(Hu)_{cf}\rangle-\langle Mu_{df},(Hu)_{df}\rangle+\langle Mu_{cf}, (L_hu)_{cf}\rangle+\langle Mu_{df}, (L_hu)_{df}\rangle\nonumber\\
&=-\langle Mu, Hu\rangle+\langle Mu_{df},(Hu)_{cf}\rangle+\langle Mu_{cf}, (Hu)_{df}\rangle\nonumber\\
&~~~+\langle Mu_{cf}, (L_hu)_{cf}\rangle+\langle Mu_{df}, (L_hu)_{df}\rangle.
\end{align}
By the~Leibniz rule, we can establish the following differential identity
\begin{align}
\langle Mu_{cf}, \Box_{c_1}u_{cf}\rangle=\partial_t{e_1}+\nabla\cdot {p_1}+q_1,
\end{align}
where
\begin{align}\label{e111}
&e_1=f(r)(\partial_t u^{i}_{cf}) (\partial_r u^{i}_{cf}+\frac{1}{r}{u^{i}_{cf}}),\\
&p_1=\frac{1}{2}f(r){\omega}(c_1^2|\nabla u_{cf}|^2-|\partial_t u_{cf}|^2)-c_1^2f(r)(\nabla u_{cf}^{i})( \partial_r u_{cf}^{i}+\frac{1}{r}{u_{cf}^{i}})\nonumber\\
&~~~~~~~+\frac{1}{2}c_1^2\frac{rf'(r)-f(r)}{r^2}{\omega}|u_{cf}|^2,\\\label{q1}
&q_1=\frac{f'(r)}{2}|\partial_t u_{cf}|^2+{c_1^2}\frac{f'(r)}{2}|\partial_r u_{cf}|^2\nonumber\\
&~~~~~~~+{c_1^2}(\frac{f(r)}{r}-\frac{f'(r)}{2})|\nabla_{\omega}  u_{cf}|^2-\frac{1}{2}c_1^2(\Delta \frac{f(r)}{r})|u_{cf}|^2,
\end{align}
and
\begin{align}
|\nabla_{\omega}  u_{cf}|^2=|\nabla  u_{cf}|^2-|\partial_r u_{cf}|^2
\end{align}
is the angular component of the gradient.
Similarly, we also have
\begin{align}
\langle Mu_{df}, \Box_{c_2}u_{df}\rangle=\partial_t{e_2}+\nabla\cdot {p_2}+q_2,
\end{align}
where
\begin{align}
&e_2=f(r)(\partial_t u^{i}_{df}) (\partial_r u^{i}_{df}+\frac{1}{r}{u^{i}_{df}}),\\
&p_2=\frac{1}{2}f(r){\omega}(c_2^2|\nabla u_{df}|^2-|\partial_t u_{df}|^2)-c_2^2f(r)(\nabla u_{df}^{i})( \partial_r u_{df}^{i}+\frac{1}{r}{u_{df}^{i}})\nonumber\\
&~~~~~~~+\frac{1}{2}c_2^2\frac{rf'(r)-f(r)}{r^2}{\omega}|u_{df}|^2,\\\label{q2}
&q_2=\frac{f'(r)}{2}|\partial_t u_{df}|^2+{c_2^2}\frac{f'(r)}{2}|\partial_r u_{df}|^2\nonumber\\
&~~~~~~~+{c_2^2}(\frac{f(r)}{r}-\frac{f'(r)}{2})|\nabla_{\omega}  u_{df}|^2-\frac{1}{2}c_2^2(\Delta \frac{f(r)}{r})|u_{df}|^2.
\end{align}
Noting the symmetry condition~\eqref{duichengxing678}, we can get
\begin{align}
\langle Mu, Hu\rangle=\nabla\cdot {p_3}+q_3,
\end{align}
where
\begin{align}
({p_3})_{l}&=f(r){\omega_k}h^{ij}_{lm}\partial_ku^{i}\partial_mu^{j}
-\frac{1}{2}f(r){\omega_l}h^{ij}_{km}\partial_ku^{i}\partial_mu^{j}\nonumber\\
&~~~~+\frac{1}{r}f(r)h^{ij}_{lm}u^{i}\partial_mu^{j},~l=1,2,\dots,n,\\\label{q3456}
q_3&=-\frac{rf'(r)-f(r)}{r}{\omega_k\omega_l}h^{ij}_{lm}\partial_ku^{i}\partial_mu^{j}
+\frac{1}{2}f'(r)h^{ij}_{lm}\partial_lu^{i}\partial_mu^{j}\nonumber\\
&~~~~+\frac{1}{2}f(r){\omega_k}(\partial_kh^{ij}_{lm})\partial_lu^{i}\partial_mu^{j}-\frac{1}{r} f(r)h^{ij}_{lm}   \partial_lu^{i}\partial_mu^{j}\nonumber\\
&~~~~-\frac{rf'(r)-f(r)}{r^2}{\omega_l} h^{ij}_{lm}  u^{i}\partial_mu^{j}.
\end{align}
We have the following differential identity
\begin{align}
\langle Mu_{df}, (Hu)_{cf}\rangle=\nabla\cdot {p_4}+q_4,
\end{align}
where
\begin{align}
{p_4}&=(Mu_{df})\nabla\cdot \Delta^{-1}(Hu),\\\label{q4biaodshi}
q_4&=-\nabla\cdot(Mu_{df}) \nabla\cdot \Delta^{-1}(Hu).
\end{align}
Noting that~$\nabla\cdot u_{df}=0$, we can verify that
\begin{align}\label{vanish}
\nabla\cdot(Mu_{df})=(f'(r)-\frac{f(r)}{r}){\omega}\cdot(\partial_ru_{df}+\frac{1}{r}u_{df}).
\end{align}
We also have the following differential identity
\begin{align}
\langle Mu_{cf}, (Hu)_{df}\rangle=\nabla\cdot {p_5}+q_5,
\end{align}
where
\begin{align}
{p_5}&= (Mu_{cf})\times \nabla\times \Delta^{-1}(Hu),\\
q_5&=-(\nabla\times (Mu_{cf}))\cdot(\nabla\times\Delta^{-1}(Hu)).
\end{align}
Noting that~$\nabla \times u_{cf}=0$, we can verify that
\begin{align}
\nabla\times(Mu_{cf})=(f'(r)-\frac{f(r)}{r}){\omega}\times(\partial_ru_{cf}+\frac{1}{r}u_{cf}).
\end{align}
\par
By the above discussion, we have
\begin{align}\label{liangduan000}
q_1+q_2&=-\partial_t(e_1+e_2)+\nabla\cdot(-p_1-p_2-p_3+p_4+p_5)\nonumber\\
&~~~-q_3+q_4+q_5+\langle Mu_{cf}, (L_hu)_{cf}\rangle+\langle Mu_{df}, (L_hu)_{df}\rangle.
\end{align}
Following \cite{MR2919103}, we take
\begin{align}\label{ggggh23457}
f(r)=\Big(\frac{r}{1+r}\Big)^{2\delta}.
\end{align}
Some simple computations can give
\begin{align}\label{ggggh2hhhh3457}
&f'(r)=2\delta r^{2\delta-1}(1+r)^{-2\delta-1},\frac{f(r)}{r}-\frac{f'(r)}{2}=r^{2\delta-1}(1+r)^{-2\delta}\Big(1-\frac{\delta}{1+r}\Big),\\\label{ggggh2hhhhffgg3457}
&\Delta\frac{f(r)}{r}\leq -2\delta(1-2\delta)r^{2\delta-3}(1+r)^{-2\delta-2},\\\label{245555ggggh2hhhh3457}
&f'(r)-\frac{f(r)}{r}=r^{2\delta-1}(1+r)^{-2\delta}\Big(\frac{2\delta}{1+r}-1\Big).
\end{align}
By \eqref{q1}, \eqref{q2}, \eqref{ggggh2hhhh3457} and \eqref{ggggh2hhhhffgg3457} , we have
\begin{align}\label{liangduan}
q_1+q_2&\geq C\big(r^{2\delta-1}(1+r)^{-2\delta-1}|\partial_t u_{cf}|^2+
r^{2\delta-1}(1+r)^{-2\delta-1}|\partial_r u_{cf}|^2\nonumber\\
&~~+r^{2\delta-1}(1+r)^{-2\delta}|\nabla_{\omega} u_{cf}|^2
+r^{2\delta-3}(1+r)^{-2\delta-2}|u_{cf}|^2\big)\nonumber\\
&~~+C\big(r^{2\delta-1}(1+r)^{-2\delta-1}|\partial_t u_{df}|^2+
r^{2\delta-1}(1+r)^{-2\delta-1}|\partial_r u_{df}|^2\nonumber\\
&~~+r^{2\delta-1}(1+r)^{-2\delta}|\nabla_{\omega} u_{df}|^2
+r^{2\delta-3}(1+r)^{-2\delta-2}|u_{df}|^2\big).
\end{align}
On both sides of \eqref{liangduan000}, we integrate over the strip~$S_T=[0,T]\times \mathbb{R}^3$. By the divergence theorem, in view of \eqref{liangduan}, we get
\begin{align}\label{liangduajkjujn}
&\int_0^{T}\!\!\!\int_{\mathbb{R}^3}~r^{2\delta-1}(1+r)^{-2\delta-1}|\partial_t u_{cf}|^2+
r^{2\delta-1}(1+r)^{-2\delta-1}|\partial_r u_{cf}|^2\nonumber\\
&~~+r^{2\delta-1}(1+r)^{-2\delta}|\nabla_{\omega} u_{cf}|^2
+r^{2\delta-3}(1+r)^{-2\delta-2}|u_{cf}|^2 ~dxdt\nonumber\\
&+\int_0^{T}\!\!\!\int_{\mathbb{R}^3}~r^{2\delta-1}(1+r)^{-2\delta-1}|\partial_t u_{df}|^2+
r^{2\delta-1}(1+r)^{-2\delta-1}|\partial_r u_{df}|^2\nonumber\\
&~~+r^{2\delta-1}(1+r)^{-2\delta}|\nabla_{\omega} u_{df}|^2
+r^{2\delta-3}(1+r)^{-2\delta-2}|u_{df}|^2 ~dxdt\nonumber\\
&\leq C_0\Big(\sup_{0\leq t\leq T}\int_{\mathbb{R}^3}|e_1|dx+\sup_{0\leq t\leq T}\int_{\mathbb{R}^3}|e_2|dx+\int_0^{T}\!\!\!\int_{\mathbb{R}^3}|q_3|dxdt\nonumber\\
&+\int_0^{T}\!\!\!\int_{\mathbb{R}^3}|q_4|dxdt
+\int_0^{T}\!\!\!\int_{\mathbb{R}^3}|q_5|dxdt\nonumber\\
&+\int_0^{T}\!\!\!\int_{\mathbb{R}^3}\langle Mu_{cf}, (L_hu)_{cf}\rangle dxdt+\int_0^{T}\!\!\!\int_{\mathbb{R}^3}\langle Mu_{df}, (L_hu)_{df}\rangle dxdt\Big).
\end{align}
\par
Now we will estimate all terms on the right-hand side of~\eqref{liangduajkjujn}. First, by \eqref{e111}, \eqref{ggggh23457}, the H\"{o}lder inequality, the Hardy inequality, \eqref{hodge456} and \eqref{hodge456789}, we have
\begin{align}
&\int_{\mathbb{R}^3}|e_1|dx\nonumber\\
&\leq C\int_{\mathbb{R}^3}|\partial_t u_{df}||\partial_r u_{df}|dx +C
\int_{\mathbb{R}^3}|\partial_t u_{df}| \frac{{|u_{df}|}}{r}dx\nonumber\\
&\leq C\|\partial_t u_{df}\|_{L^2(\mathbb{R}^{3})} \|\nabla u_{df}\|_{L^2(\mathbb{R}^{3})}
+C\|\partial_t u_{df}\|_{L^2(\mathbb{R}^{3})}  \big\|\frac{{u_{df}}}{r}\big\|_{L^2(\mathbb{R}^{3})}\nonumber\\
&\leq C\|\partial_t u_{df}\|_{L^2(\mathbb{R}^{3})} \|\nabla u_{df}\|_{L^2(\mathbb{R}^{3})}\leq C\|\partial_t u\|_{L^2(\mathbb{R}^{3})} \|\nabla u\|_{L^2(\mathbb{R}^{3})}\leq C\|\partial u\|^2_{L^2(\mathbb{R}^{3})} .
\end{align}
Similarly, we also have
\begin{align}
\int_{\mathbb{R}^3}|e_2|dx\leq C\|\partial u\|^2_{L^2(\mathbb{R}^{3})} .
\end{align}
So it holds that
\begin{align}\label{141234}
\sup_{0\leq t\leq T}\int_{\mathbb{R}^3}|e_1|dx+\sup_{0\leq t\leq T}\int_{\mathbb{R}^3}|e_2|dx\leq  C\sup_{0\leq t\leq T}\|\partial u\|^2_{L^2(\mathbb{R}^{3})}.
\end{align}
By~\eqref{q3456},~\eqref{ggggh2hhhh3457} and~\eqref{245555ggggh2hhhh3457}, we have the following pointwise estimate
\begin{align}\label{oubjer56}
|q_3|\leq C|\nabla h||\nabla u|^2+Cr^{2\delta}(1+r)^{-2\delta}\frac{|h|}{r}(\frac{|u|}{r}+|\nabla u|)|\nabla u|.
\end{align}
Hence by \eqref{oubjer56}, the H\"{o}lder inequality, the Hardy inequality and the Cauchy-Schwarz inequality, we have
\begin{align}\label{q3nbouy}
&\int_0^{T}\!\!\!\int_{\mathbb{R}^3}|q_3|dxdt\nonumber\\
&\leq C\int_0^{T}\!\!\!\int_{\mathbb{R}^3}|\nabla h||\nabla u|^2+r^{2\delta}(1+r)^{-2\delta}\frac{|h|}{r}(\frac{|u|}{r}+|\nabla u|)|\nabla u| dxdt\nonumber\\
&\leq C\sup_{0\leq t\leq T}\|\nabla u\|_{L^2(\mathbb{R}^{3})}\||\nabla h||\nabla u|\|_{L^1_tL^2_{x}(S_T)}\nonumber\\
&+C\|\langle r\rangle^{-\delta}{ r}^{-{1}/{2}+\delta}{|h|}|\nabla u|\|_{L^2_{t,x}(S_{T})}\|\langle r\rangle^{-\delta}{ r}^{-{1}/{2}+\delta}|\nabla u|\|_{L^2_{t,x}(S_{T})}
\nonumber\\
&+C\|\langle r\rangle^{-\delta}{ r}^{-{1}/{2}+\delta}{|h|}|\nabla u|\|_{L^2_{t,x}(S_{T})}\|\langle r\rangle^{-\delta}{ r}^{-{3}/{2}+\delta}| u|\|_{L^2_{t,x}(S_{T})}\nonumber\\
&\leq C\sup_{0\leq t\leq T}\|\nabla u\|^2_{L^2(\mathbb{R}^{3})}+C\||\nabla h||\nabla u|\|^2_{L^1_tL^2_{x}(S_T)}\nonumber\\
&+C\|\langle r\rangle^{-\delta}{ r}^{-{1}/{2}+\delta}{|h|}|\nabla u|\|_{L^2_{t,x}(S_{T})}\|\langle r\rangle^{-\delta}{ r}^{-{1}/{2}+\delta}|\nabla u|\|_{L^2_{t,x}(S_{T})}
\nonumber\\
&+C\|\langle r\rangle^{-\delta}{ r}^{-{1}/{2}+\delta}{|h|}|\nabla u|\|_{L^2_{t,x}(S_{T})}\|\langle r\rangle^{-\delta}{ r}^{-{3}/{2}+\delta}| u|\|_{L^2_{t,x}(S_{T})}.
\end{align}
By~\eqref{Hdefin}, \eqref{q4biaodshi}, \eqref{vanish}, \eqref{245555ggggh2hhhh3457}, the H\"{o}lder inequality, the Hardy inequality and Lemma \ref{yaoyaohj},  we have
\begin{align}\label{shangj}
&C_0\int_{\mathbb{R}^3}|q_4|dx\nonumber\\
& \leq C\int_{\mathbb{R}^3}\big(r^{2\delta-1}(1+r)^{-2\delta}|\partial_ru_{df}|+r^{2\delta-2}(1+r)^{-2\delta}|u_{df}|\big)
\big|\Delta^{-1}\partial_i\partial_l(h^{ij}_{lm}\partial_mu^{j})\big|dx\nonumber\\
& \leq C\int_{\mathbb{R}^3}(|\langle r\rangle^{-\delta}r^{-1/2+\delta}\partial_ru_{df}|+{|\langle r\rangle^{-\delta}r^{-3/2+\delta}u_{df}|})
\big|\langle r\rangle^{-\delta}r^{-1/2+\delta}\Delta^{-1} \partial_i\partial_l(h^{ij}_{lm}\partial_mu^{j})\big|dx\nonumber\\
& \leq C\big(\|\langle r\rangle^{-\delta}r^{-1/2+\delta}\partial_r u_{df}\|_{L^2(\mathbb{R}^{3})}+\big\|\langle r\rangle^{-\delta}r^{-3/2+\delta}{u_{df}}\big\|_{L^2(\mathbb{R}^{3})}\big)
\| \langle r\rangle^{-\delta}r^{-1/2+\delta}R_iR_l(h^{ij}_{lm}\partial_mu^{j})\|_{L^2(\mathbb{R}^3)}\nonumber\\
& \leq C\big(\|\langle r\rangle^{-\delta}r^{-1/2+\delta}\partial u_{df}\|_{L^2(\mathbb{R}^{3})}+\big\|\langle r\rangle^{-\delta}r^{-3/2+\delta}{u_{df}}\big\|_{L^2(\mathbb{R}^{3})}\big)
\| \langle r\rangle^{-\delta}r^{-1/2+\delta}|h||\nabla u|\|_{L^2(\mathbb{R}^3)}\nonumber\\
& \leq \frac{1}{100}(\log(2+T))^{-1}\|\langle r\rangle^{-\delta}r^{-1/2+\delta}\partial u_{df}\|^2_{L^2(\mathbb{R}^{3})}+\frac{1}{100}(\log(2+T))^{-1}\big\|\langle r\rangle^{-\delta}r^{-3/2+\delta}{u_{df}}\big\|^2_{L^2(\mathbb{R}^{3})}\nonumber\\
&+C(\log(2+T))
\| \langle r\rangle^{-\delta}r^{-1/2+\delta}|h||\nabla u|\|^2_{L^2(\mathbb{R}^3)}.
\end{align}
Then we have
\begin{align}\label{sklghhghjjlangjrtt}
&C_0\int_0^{T}\!\!\!\int_{\mathbb{R}^3}|q_4| dxdt\nonumber\\
& \leq \frac{1}{100}(\log(2+T))^{-1}\|\langle r\rangle^{-\delta}r^{-1/2+\delta}\partial u_{df}\|^2_{L^2_{t,x}(S_T)}+\frac{1}{100}(\log(2+T))^{-1}\big\|\langle r\rangle^{-\delta}r^{-3/2+\delta}{u_{df}}\big\|^2_{L^2_{t,x}(S_T)}\nonumber\\
&+C(\log(2+T))
\| \langle r\rangle^{-\delta}r^{-1/2+\delta}|h||\nabla u|\|^2_{L^2_{t,x}(S_T)}.
\end{align}
 Similarly, we can show that
\begin{align}\label{skllangj}
&C_0\int_0^{T}\!\!\!\int_{\mathbb{R}^3}|q_5| dxdt\nonumber\\
& \leq \frac{1}{100}(\log(2+T))^{-1}\|\langle r\rangle^{-\delta}r^{-1/2+\delta}\partial u_{cf}\|^2_{L^2_{t,x}(S_T)}+\frac{1}{100}(\log(2+T))^{-1}\big\|\langle r\rangle^{-\delta}r^{-3/2+\delta}{u_{cf}}\big\|^2_{L^2_{t,x}(S_T)}\nonumber\\
&+C(\log(2+T))
\| \langle r\rangle^{-\delta}r^{-1/2+\delta}|h||\nabla u|\|^2_{L^2_{t,x}(S_T)}.
\end{align}
It follows from  \eqref{multiplier}, \eqref{ggggh23457}, the H\"{o}lder inequality, the Hardy inequality and \eqref{hodge456789} that
\begin{align}\label{sklghhlangj}
&\int_{\mathbb{R}^3}\langle Mu_{cf}, (L_hu)_{cf}\rangle dx\nonumber\\
&\leq \int_{\mathbb{R}^3}(|\partial_r u_{cf}|+\frac{|u_{cf}|}{r})|(L_hu)_{cf}|dxdt\nonumber\\
&\leq \big(\|\partial_r u_{cf}\|_{L^2(\mathbb{R}^3)}+\big\|\frac{u_{cf}}{r}\big\|_{L^2(\mathbb{R}^3)}\big)\|(L_hu)_{cf}\|_{L^2(\mathbb{R}^3)}\nonumber\\
&\leq C\|\nabla u_{cf}\|_{L^2(\mathbb{R}^3)}\|(L_hu)_{cf}\|_{L^2(\mathbb{R}^3)}\leq C\|\nabla u\|_{L^2(\mathbb{R}^3)}\|L_hu\|_{L^2(\mathbb{R}^3)}.
\end{align}
Then by the H\"{o}lder inequality and the Cauchy-Schwarz inequality, we have
\begin{align}\label{sklghhghjjlangj}
&\int_0^{T}\!\!\!\int_{\mathbb{R}^3}\langle Mu_{cf}, (L_hu)_{cf}\rangle dxdt\nonumber\\
&\leq C\sup_{0\leq t\leq T}\|\nabla u\|_{L^2(\mathbb{R}^3)}\|L_hu\|_{L^1_{t}L^2_{x}(S_{T})}\leq C\sup_{0\leq t\leq T}\|\nabla u\|^2_{L^2(\mathbb{R}^3)}+C\|L_hu\|^2_{L^1_{t}L^2_{x}(S_{T})}.
\end{align}
Similarly, we also have
\begin{align}\label{sklghuiihghjjlangj}
&\int_0^{T}\!\!\!\int_{\mathbb{R}^3}\langle Mu_{df}, (L_hu)_{df}\rangle dxdt\leq C\sup_{0\leq t\leq T}\|\nabla u\|^2_{L^2(\mathbb{R}^3)}+C\|L_hu\|^2_{L^1_{t}L^2_{x}(S_{T})}.
\end{align}
\par
By \eqref{liangduajkjujn}, \eqref{141234}, \eqref{q3nbouy}, \eqref{sklghhghjjlangjrtt}--\eqref{sklghuiihghjjlangj}, we can get
\begin{align}\label{lrtyui}
&\|r^{-1/2+\delta}\partial u_{cf}\|^2_{L^2(0,T;L^2(|x|\leq 1))}+
\|r^{-3/2+\delta} u_{cf}\|^2_{L^2(0,T;L^2(|x|\leq 1))}\nonumber\\
&+\|r^{-1/2+\delta}\partial u_{df}\|^2_{L^2(0,T;L^2(|x|\leq 1))}+
\|r^{-3/2+\delta} u_{df}\|^2_{L^2(0,T;L^2(|x|\leq 1))}\nonumber\\
&\leq C\|\partial u(0,\cdot)\|^2_{L^2(\mathbb{R}^{3})}+
C\|L_hu\|^2_{L^1_{t}L^2_{x}(S_{T})}+C\||\partial h||\nabla u|\|^2_{L^1_tL^2_{x}(S_T)}\nonumber\\
&+C\|\langle r\rangle^{-\delta}{ r}^{-{1}/{2}+\delta}{|h|}|\nabla u|\|_{L^2_{t,x}(S_{T})}\big(\|\langle r\rangle^{-\delta}{ r}^{-{1}/{2}+\delta}|\nabla u|\|_{L^2_{t,x}(S_{T})}
+\|\langle r\rangle^{-\delta}{ r}^{-{3}/{2}+\delta}| u|\|_{L^2_{t,x}(S_{T})}\big)\nonumber\\
&+\frac{1}{100}(\log(2+T))^{-1}\|\langle r\rangle^{-\delta}r^{-1/2+\delta}\partial u_{cf}\|^2_{L^2_{t,x}(S_T)}+\frac{1}{100}(\log(2+T))^{-1}\big\|\langle r\rangle^{-\delta}r^{-3/2+\delta}{u_{cf}}\big\|^2_{L^2_{t,x}(S_T)}\nonumber\\
&+C(\log(2+T))
\| \langle r\rangle^{-\delta}r^{-1/2+\delta}|h||\nabla u|\|^2_{L^2_{t,x}(S_T)}\nonumber\\
&+\frac{1}{100}(\log(2+T))^{-1}\|\langle r\rangle^{-\delta}r^{-1/2+\delta}\partial u_{df}\|^2_{L^2_{t,x}(S_T)}+\frac{1}{100}(\log(2+T))^{-1}\big\|\langle r\rangle^{-\delta}r^{-3/2+\delta}{u_{df}}\big\|^2_{L^2_{t,x}(S_T)}.
\end{align}
The combination of \eqref{KSS00} and \eqref{lrtyui} gives
\begin{align}\label{rhs345}
&(\log(2+T))^{-1}\|\langle r\rangle^{-\delta}r^{-1/2+\delta}\partial u_{cf}\|^2_{L^2_{t,x}(S_{T})}
+(\log(2+T))^{-1}\|\langle r\rangle^{-\delta}r^{-3/2+\delta} u_{cf}\|^2_{L^2_{t,x}(S_{T})}\nonumber\\
&+
(\log(2+T))^{-1}\|\langle r\rangle^{-\delta}r^{-1/2+\delta}\partial u_{df}\|^2_{L^2_{t,x}(S_{T})}
+(\log(2+T))^{-1}\|\langle r\rangle^{-\delta}r^{-3/2+\delta} u_{df}\|^2_{L^2_{t,x}(S_{T})}\nonumber\\
&\leq C\|\partial u(0,\cdot)\|^2_{L^2(\mathbb{R}^{3})}+
C\|L_hu\|^2_{L^1_{t}L^2_{x}(S_{T})}+C\||\partial h||\nabla u|\|^2_{L^1_tL^2_{x}(S_T)}+C\|{\langle r\rangle}^{-1}{|h|}|\nabla u|\|^2_{L^1_{t}L^2_{x}(S_{T})}\nonumber\\
&+C\|\langle r\rangle^{-\delta}{ r}^{-{1}/{2}+\delta}{|h|}|\nabla u|\|_{L^2_{t,x}(S_{T})}\big(\|\langle r\rangle^{-\delta}{ r}^{-{1}/{2}+\delta}|\nabla u|\|_{L^2_{t,x}(S_{T})}
+\|\langle r\rangle^{-\delta}{ r}^{-{3}/{2}+\delta}| u|\|_{L^2_{t,x}(S_{T})}\big)\nonumber\\
&+C(\log(2+T))
\| \langle r\rangle^{-\delta}r^{-1/2+\delta}|h||\nabla u|\|^2_{L^2_{t,x}(S_T)}\nonumber\\
&+\frac{1}{100}(\log(2+T))^{-1}\|\langle r\rangle^{-\delta}r^{-1/2+\delta}\partial u_{cf}\|^2_{L^2_{t,x}(S_T)}+\frac{1}{100}(\log(2+T))^{-1}\big\|\langle r\rangle^{-\delta}r^{-3/2+\delta}{u_{cf}}\big\|^2_{L^2_{t,x}(S_T)}\nonumber\\
&+\frac{1}{100}(\log(2+T))^{-1}\|\langle r\rangle^{-\delta}r^{-1/2+\delta}\partial u_{df}\|^2_{L^2_{t,x}(S_T)}+\frac{1}{100}(\log(2+T))^{-1}\big\|\langle r\rangle^{-\delta}r^{-3/2+\delta}{u_{df}}\big\|^2_{L^2_{t,x}(S_T)}
.
\end{align}
The last four terms on the right-hand side of \eqref{rhs345} can be absorbed into the the left-hand side of \eqref{rhs345}. By \eqref{hodge1}, it is obvious that
\begin{align}\label{com2}
|u|\leq |u_{cf}|+|u_{df}|,
\end{align}
which implies that
\begin{align}\label{dfrgg}
|u|^2\leq |u_{cf}|^2+|u_{df}|^2+2|u_{cf}||u_{df}|\leq 2(|u_{cf}|^2+|u_{df}|^2).
\end{align}
Hence it follows from \eqref{dfrgg} and the Cauchy-Schwarz inequality that
\begin{align}\label{rhs34578tyy}
&(\log(2+T))^{-1}\|\langle r\rangle^{-\delta}r^{-1/2+\delta}\partial u\|^2_{L^2_{t,x}(S_{T})}
+(\log(2+T))^{-1}\|\langle r\rangle^{-\delta}r^{-3/2+\delta} u\|^2_{L^2_{t,x}(S_{T})}\nonumber\\
&\leq C\|\partial u(0,\cdot)\|^2_{L^2(\mathbb{R}^{3})}+
C\|L_hu\|^2_{L^1_{t}L^2_{x}(S_{T})}+C\||\partial h||\nabla u|\|^2_{L^1_tL^2_{x}(S_T)}+C\|{\langle r\rangle}^{-1}{|h|}|\nabla u|\|^2_{L^1_{t}L^2_{x}(S_{T})}\nonumber\\
&+C\|\langle r\rangle^{-\delta}{ r}^{-{1}/{2}+\delta}{|h|}|\nabla u|\|_{L^2_{t,x}(S_{T})}\big(\|\langle r\rangle^{-\delta}{ r}^{-{1}/{2}+\delta}|\nabla u|\|_{L^2_{t,x}(S_{T})}
+\|\langle r\rangle^{-\delta}{ r}^{-{3}/{2}+\delta}| u|\|_{L^2_{t,x}(S_{T})}\big)\nonumber\\
&+C(\log(2+T))
\| \langle r\rangle^{-\delta}r^{-1/2+\delta}|h||\nabla u|\|^2_{L^2_{t,x}(S_T)}\nonumber\\
&\leq C\|\partial u(0,\cdot)\|^2_{L^2(\mathbb{R}^{3})}+
C\|L_hu\|^2_{L^1_{t}L^2_{x}(S_{T})}+C\||\partial h||\nabla u|\|^2_{L^1_tL^2_{x}(S_T)}+C\|{\langle r\rangle}^{-1}{|h|}|\nabla u|\|^2_{L^1_{t}L^2_{x}(S_{T})}\nonumber\\
&+C(\log(2+T))
\| \langle r\rangle^{-\delta}r^{-1/2+\delta}|h||\nabla u|\|^2_{L^2_{t,x}(S_T)}\nonumber\\
&+\frac{1}{100}(\log(2+T))^{-1}\|\langle r\rangle^{-\delta}r^{-1/2+\delta}\partial u\|^2_{L^2_{t,x}(S_T)}+\frac{1}{100}(\log(2+T))^{-1}\big\|\langle r\rangle^{-\delta}r^{-3/2+\delta}{u}\big\|^2_{L^2_{t,x}(S_T)}
.
\end{align}
Note that the last two terms on the right-hand side of \eqref{rhs34578tyy} can be absorbed into the the left-hand side of \eqref{rhs34578tyy}. Using also the energy inequality \eqref{basicene}, we can get
\begin{align}\label{KSSfinal}
&\sup_{0\leq t\leq T}\|\partial u\|_{L^2(\mathbb{R}^{3})}+ (\log(2+T))^{-1/2}\|\langle r\rangle^{-\delta}r^{-1/2+\delta}\partial u\|_{L^2_{t,x}(S_{T})}\nonumber\\
&~~~~~~~~~~~~~~~~~~~~~+
(\log(2+T))^{-1/2}\|\langle r\rangle^{-\delta}r^{-3/2+\delta} u\|_{L^2_{t,x}(S_{T})}\nonumber\\
&\leq C\|\partial u(0,\cdot)\|_{L^2(\mathbb{R}^{3})}+
C\|L_hu\|_{L^1_{t}L^2_{x}(S_{T})}+C\||\partial h||\nabla u|\|_{L^1_{t}L^2_{x}(S_{T})}\nonumber\\
&~~+C\|{\langle r\rangle}^{-1}{|h|}|\nabla u|\|_{L^1_{t}L^2_{x}(S_{T})}
+C(\log(2+T))^{1/2}\|\langle r\rangle^{-\delta}{ r}^{-{1}/{2}+\delta}{|h|}|\nabla u|\|_{L^2_{t,x}(S_{T})}.
\end{align}
Now we complete the proof of Theorem \ref{thmKSS00}.
\section{Proof of Theorem \ref{mainthm}}\label{sec3}
The aim of this section is to prove Theorem \ref{mainthm} concerning the almost global existence for nonlinear elastic waves, which requires the smallness of data with respect to some weighted $H^4\times H^3$ norm. Our strategy is to combine the KSS type estimate \eqref{KSS111} and
some weight Sobolev inequalities (Lemma \ref{dddd4444}), in which the angular integrability often plays a key role.
\par
We will use the following mixed-norm
\begin{align}
\|u\|_{L^p_rL^{q}_{\omega}}:=\Big(\int_{0}^{+\infty}\|u(rw)\|_{L^q_{\omega}(S^2)}^{p}r^2dr\Big)^{1/p},
\end{align}
with trivial modification for the case $p=+\infty$. \par
\begin{Lemma}\label{dddd4444}
For~$u\in C_{0}^{\infty}(\mathbb{R}^{3};\mathbb{R}^{3})$, we have
\begin{align}\label{sharp345}
\langle r\rangle^{1/2}|u(x)|&\leq C\sum_{|a|\leq 2}\|\langle r\rangle^{-1/2}Z^{a}u\|_{L^2(\mathbb{R}^{3})},\\\label{conseq}
\|u\|_{L^2_rL^{q}_{\omega}}&\leq C\sum_{|a|\leq 1}\|\widetilde{\Omega}^{a}u\|_{L^2(\mathbb{R}^{3})},~2\leq q<+\infty,\\\label{jiaodao}
\|r^{1/2}u\|_{L^{\infty}_rL^{p}_{\omega}(r\geq 1)}&\leq C\sum_{|a|\leq 1}\|\langle r\rangle^{-1/2}Z^{a}u\|_{L^2(\mathbb{R}^{3})},~2\leq p\leq 4,\\\label{sharp123}
\|r^{(3/2)-s}u\|_{L^{\infty}_rL^{p}_{\omega}}&\leq C\|u\|_{{H}^{1}(\mathbb{R}^{3})}, ~\frac{1}{2}<s\leq 1,~\frac{2}{p}=-s+\frac{3}{2}.
\end{align}
\end{Lemma}
\begin{proof}
It follows from the Sobolev embedding $H^2(B_1)\hookrightarrow L^{\infty}(B_1)$ and the following weighted Sobolev inequality (see (3.14b) of \cite{Sideris00})
\begin{align}\label{hjkiile}
r|u(x)|\leq C\sum_{|a|\leq 1}\|\partial_r\widetilde{\Omega}^{a}u\|_{L^2(|y|\geq r)}^{1/2}\sum_{|a|\leq 2}\|\widetilde{\Omega}^{a}u\|_{L^2(|y|\geq r)}^{1/2}
\end{align}
that
\begin{align}\label{replace}
\langle r\rangle|u(x)|&\leq C\sum_{|a|\leq 2}\|Z^{a}u\|_{L^2(\mathbb{R}^{3})}.
\end{align}
\eqref{sharp345} can be proved by replacing $u$ by $\langle r\rangle^{-1/2} u$ in \eqref{replace}.
\eqref{conseq} is a consequence of the Sobolev embedding on $S^2: H^1(S^2)\hookrightarrow L^{q}(S^2), 2\leq q<+\infty$.
As for \eqref{jiaodao}, we only need to prove the case $p=4$. By (3.19) of Sideris \cite{Sideris00}, we have
\begin{align}\label{hjkiile888}
r\|u(r\omega)\|_{L^4_{\omega}}\leq C\|\partial_ru\|_{L^2(|y|\geq r)}^{1/2}\sum_{|a|\leq 1}\|\widetilde{\Omega}^{a}u\|_{L^2(|y|\geq r)}^{1/2}.
\end{align}
Replacing $u$ by $r^{-1/2}u$ in \eqref{hjkiile888}, we can get
\begin{align}
\|r^{1/2}u\|_{L^{\infty}_rL^{4}_{\omega}(r\geq 1)}&\leq C\sum_{|a|\leq 1}\|\langle r\rangle^{-1/2}Z^{a}u\|_{L^2(\mathbb{R}^{3})}.
\end{align}
For \eqref{sharp123}, it follows from the trace inequality due to Hoshiro \cite{MR1482992} and the Sobolev embedding on $S^2$ that
\begin{align}\label{kkljk}
r^{({3}/{2})-s}\|u(r\omega)\|_{L^{p}_{\omega}}\leq C\|u\|_{\dot{H}^{s}(\mathbb{R}^{3})}, ~\frac{1}{2}<s<\frac{3}{2},~\frac{2}{p}=-s+\frac{3}{2}.
\end{align}
See, e.g., Proposition 2.4 of \cite{MR3552253}. \eqref{sharp123} is a consequence of \eqref{kkljk}
and the interpolational inequality:
\begin{align}
\|u\|_{\dot{H}^{s}(\mathbb{R}^{3})}\leq C\|u\|^{1-s}_{L^2(\mathbb{R}^{3})}\|u\|^{s}_{\dot{H}^{1}(\mathbb{R}^{3})}\leq C\|u\|_{{H}^{1}(\mathbb{R}^{3})},~~0<s\leq 1.
\end{align}
\end{proof}

 As Proposition 3.1 in~Sideris \cite{Sideris00}, we have the following
\begin{Lemma}\label{Nulll}
For any solution $u$~of \eqref{Cauchy} in ${X}^{k}_{Z}(T)$,
we have
\begin{align}
LZ^{a}u=\sum_{b+c=a}N(Z^{b}u,Z^{c}u),
\end{align}
in which the sum extends over ordered partitions of the sequences $a$, with $|a|\le k-1$.
\end{Lemma}
Now we will prove Theorem \ref{mainthm}. Assume that $u=u(t,x)$ is a local solution of the Cauchy problem \eqref{Cauchy} on $[0,T]$.
 We will show that there exist positive constants $\kappa$, $\varepsilon_0$ and $A$ such that for any $T\leq \exp({\kappa}/{\varepsilon})$, we have $\|u\|_{X^{4}_{Z}(T)}\leq  A \varepsilon$ under the assumption that \eqref{label345} and $\|u\|_{X^{4}_{Z}(T)}\leq 2A \varepsilon$, where $0<\varepsilon\leq\varepsilon_0$.\par
 By Lemma \ref{Nulll} and \eqref{KSS111} in Theorem \ref{thmKSS00}, we can get
 \begin{align}\label{byby1}
 \|u\|_{X^{4}_{Z}(T)}\leq &C_1\varepsilon+C\sum_{|a|\leq 3}\|\partial \nabla u\nabla Z^au\|_{L^1_tL^2_x(S_T)}+C\sum_{|a|\leq 3}\| \nabla u\nabla Z^au\|_{L^1_tL^2_x(S_T)}\nonumber\\
 &+C(\log(2+T))^{1/2}\sum_{|a|\leq 3}\|\langle r\rangle^{-\delta}{ r}^{-{1}/{2}+\delta}{\nabla u}\nabla Z^{a}u\|_{L^2_{t,x}(S_{T})}\nonumber\\
 &+C\sum_{|a|\leq 3}\sum_{\substack{b+c=a\\b,c\neq a}}\|\nabla Z^bu\nabla^2 Z^cu\|_{L^1_tL^2_x(S_T)}.
 \end{align}
 For $|a|\leq 3$, by the H\"{o}lder inequality and \eqref{sharp345}, we can get
 \begin{align}\label{similar}
& \||\partial \nabla u\nabla Z^au\|_{L^1_tL^2_x(S_T)}\nonumber\\
& \leq \|\langle r\rangle^{-1/2}\nabla Z^au\|_{L^2_{t,x}(S_T)}\|\langle r\rangle^{1/2}\partial \nabla u\|_{L^2_tL^{\infty}_x(S_T)}\nonumber\\
& \leq C\|\langle r\rangle^{-1/2}\nabla Z^au\|_{L^2_{t,x}(S_T)}\sum_{|\alpha|\leq 3}\|\langle r\rangle^{-1/2}\partial Z^{\alpha}u\|_{L^2_{t,x}(S_T)}\nonumber\\
& \leq C\sum_{|\alpha|\leq 3}\|\langle r\rangle^{-\delta}r^{-1/2+\delta}\partial Z^{\alpha}u\|^2_{L^2_{t,x}(S_T)}
\nonumber\\
& \leq C (\log(2+T))\|u\|_{X^{4}_{Z}(T)}^2.
 \end{align}
 Similarly, it also holds that for $|a|\leq 3$,
 \begin{align}
 \|\nabla u\nabla Z^au\|_{L^1_tL^2_x(S_T)}
 \leq C (\log(2+T))\|u\|_{X^{4}_{Z}(T)}^2.
 \end{align}
  For $|a|\leq 3$, it follows from the H\"{o}lder inequality and the Sobolev embedding $H^2(\mathbb{R}^{3})\hookrightarrow L^{\infty}(\mathbb{R}^{3})$ that
 \begin{align}
 &(\log(2+T))^{1/2}\|\langle r\rangle^{-\delta}{ r}^{-{1}/{2}+\delta}{\nabla u}\nabla Z^{a}u\|_{L^2_{t,x}(S_{T})}\nonumber\\
 &\leq (\log(2+T))^{1/2}\|\langle r\rangle^{-\delta}{ r}^{-{1}/{2}+\delta}\nabla Z^{a}u\|_{L^2_{t,x}(S_{T})}\|{\nabla u}\|_{L^{\infty}_{t,x}(S_{T})}\nonumber\\
 &\leq C(\log(2+T))^{1/2}\|\langle r\rangle^{-\delta}{ r}^{-{1}/{2}+\delta}\nabla Z^{a}u\|_{L^2_{t,x}(S_{T})}\sum_{|\alpha|\leq 2}\|\nabla Z^{\alpha}u\|_{L^{\infty}_{t}L^2_{x}(S_T)}\nonumber\\
 &\leq C (\log(2+T))\|u\|_{X^{4}_{Z}(T)}^2.
 \end{align}
 Now  we will estimate the last term on the right-hand side of \eqref{byby1}
For $|a|\leq 3, b+c=a,b,c\neq a$, if $|b|\leq 1, |c|\leq 2$, similarly to \eqref{similar}, by the H\"{o}lder inequality and \eqref{sharp345}, we can get
 \begin{align}\label{byby4}
&\|\nabla Z^bu\nabla^2 Z^cu\|_{L^1_tL^2_x(S_T)} \nonumber\\
& \leq \|\langle r\rangle^{-1/2}\nabla^2 Z^cu\|_{L^2_{t,x}(S_T)}\|\langle r\rangle^{1/2}\nabla Z^bu\|_{L^2_tL^{\infty}_x(S_T)}\nonumber\\
&\leq C\|\langle r\rangle^{-1/2}\nabla^2 Z^cu\|_{L^2_{t,x}(S_T)}\sum_{|\alpha|\leq 3}\|\langle r\rangle^{-1/2}\nabla Z^{\alpha}u\|_{L^2_{t,x}(S_T)}\nonumber\\
& \leq C\sum_{|\alpha|\leq 3}\|\langle r\rangle^{-\delta}r^{-1/2+\delta}\nabla Z^{\alpha}u\|^2_{L^2_{t,x}(S_T)}
\nonumber\\
& \leq C (\log(2+T))\|u\|_{X^{4}_{Z}(T)}^2.
 \end{align}
If $|b|\leq 2, |c|\leq 1$, it follows from the H\"{o}lder inequality that
 \begin{align}\label{byby4567}
\|\nabla Z^bu\nabla^2 Z^cu\|_{L^1_tL^2_x(S_T)} \leq \|\langle r\rangle^{-\delta}r^{-1/2+\delta}\nabla^2 Z^cu\|_{L^2_tL^{2}_{r}L^{q}_{\omega}}\|\langle r\rangle^{\delta}r^{1/2-\delta}\nabla Z^bu\|_{L^2_tL_{r}^{\infty}L^{p}_{\omega}},
 \end{align}
where $2\leq p,q< +\infty, {1}/{p}+{1}/{q}={1}/{2}$. By \eqref{conseq}, we have
 \begin{align}\label{byby45671}
 \|\langle r\rangle^{-\delta}r^{-1/2+\delta}\nabla^2 Z^cu\|_{L^{2}_{r}L^{q}_{\omega}}\leq C\sum_{|\alpha|\leq 3}\|\langle r\rangle^{-\delta}r^{-1/2+\delta}\nabla Z^{\alpha}u\|_{L^2({\mathbb{R}^3})}.
 \end{align}
It is obvious that
 \begin{align}\label{byby45671897}
&\|\langle r\rangle^{\delta}r^{1/2-\delta}\nabla Z^bu\|_{L_{r}^{\infty}L^{p}_{\omega}}\nonumber\\
&\leq C\|\langle r\rangle^{\delta}r^{1/2-\delta}\nabla Z^bu\|_{L_{r}^{\infty}L^{p}_{\omega}(r\geq 1)}+C\|\langle r\rangle^{\delta}r^{1/2-\delta}\nabla Z^bu\|_{L_{r}^{\infty}L^{p}_{\omega}(r\leq 1)}\nonumber\\
&\leq C\|r^{1/2}\nabla Z^bu\|_{L_{r}^{\infty}L^{p}_{\omega}(r\geq 1)}+C\|r^{1/2-\delta}\nabla Z^bu\|_{L_{r}^{\infty}L^{p}_{\omega}(r\leq 1)}.
 \end{align}
 Now we take $p={2}/{(1-2\delta)}$. Noting that $0<\delta\leq 1/4$ implies that $2<p\leq 4$. By \eqref{jiaodao}, we have
 \begin{align}\label{laeluu}
 &\|r^{1/2}\nabla Z^bu\|_{L_{r}^{\infty}L^{p}_{\omega}(r\geq 1)}\nonumber\\
 &\leq C\sum_{|\alpha|\leq 3}\|\langle r\rangle^{-1/2}\nabla Z^{\alpha}u\|_{L^2({\mathbb{R}^3})}\leq C\sum_{|\alpha|\leq 3}\|\langle r\rangle^{-\delta}r^{-1/2+\delta}\nabla Z^{\alpha}u\|_{L^2({\mathbb{R}^3})}.
 \end{align}
 It follows from \eqref{sharp123} that
 \begin{align}\label{laeluuddd}
 &\|r^{1/2-\delta}\nabla Z^bu\|_{L_{r}^{\infty}L^{p}_{\omega}(r\leq 1)}\nonumber\\
 &\leq C\|r^{1/2-\delta}\langle r\rangle^{-\delta}\nabla Z^bu\|_{L_{r}^{\infty}L^{p}_{\omega}(r\leq 1)}\nonumber\\
 &\leq C\|r^{1-2\delta}\langle r\rangle^{-\delta}r^{-1/2+\delta}\nabla Z^bu\|_{L_{r}^{\infty}L^{p}_{\omega}}\leq C\|\langle r\rangle^{-\delta}r^{-1/2+\delta}\nabla Z^bu\|_{H^1({\mathbb{R}^3})}\nonumber\\
 &\leq C\sum_{|\alpha|\leq 3}\|\langle r\rangle^{-\delta}r^{-1/2+\delta}\nabla Z^{\alpha}u\|_{L^2({\mathbb{R}^3})}+C\sum_{|\alpha|\leq 3}\|\langle r\rangle^{-\delta}r^{-3/2+\delta} Z^{\alpha}u\|_{L^2({\mathbb{R}^3})}.
 \end{align}
 By \eqref{byby45671897}, \eqref{laeluu} and \eqref{laeluuddd}, we have
 \begin{align}\label{laeluuddfggd}
 &\|\langle r\rangle^{\delta}r^{1/2-\delta}\nabla Z^bu\|_{L_{r}^{\infty}L^{p}_{\omega}}\nonumber\\
 &\leq C\sum_{|\alpha|\leq 3}\|\langle r\rangle^{-\delta}r^{-1/2+\delta}\nabla Z^{\alpha}u\|_{L^2({\mathbb{R}^3})}+C\sum_{|\alpha|\leq 3}\|\langle r\rangle^{-\delta}r^{-3/2+\delta} Z^{\alpha}u\|_{L^2({\mathbb{R}^3})}.
 \end{align}
It follows from \eqref{byby4567}, \eqref{byby45671} and \eqref{laeluuddfggd} that when $|b|\leq 2, |c|\leq 1$, it holds that
 \begin{align}\label{byby4567dddddd}
\|\nabla Z^bu\nabla^2 Z^cu\|_{L^1_tL^2_x(S_T)}  \leq C (\log(2+T))\|u\|_{X^{4}_{Z}(T)}^2.
 \end{align}
 \par

Hence the above argument gives
\begin{align}
\|u\|_{X^{4}_{Z}(T)}\leq C_1\varepsilon+C_2(\log(2+T))\|u\|_{X^{4}_{Z}(T)}^2
\leq C_1\varepsilon+4C_2(\log(2+T))A^2\varepsilon^2.
\end{align}
Take $A=4C_1$ and $\varepsilon_0>0$ sufficiently small.
Then for any $0<\varepsilon\leq\varepsilon_0$, if
\begin{align}\label{ertftt}
16C_2(\log(2+T))A\varepsilon\leq 1,
\end{align}
then it holds that
\begin{align}
\|u\|_{X^{4}_{Z}(T)}\leq A\varepsilon.
\end{align}
Consequently, it follows from \eqref{ertftt} that we can get the lifespan estimate of smooth solutions to the Cauchy problem \eqref{Cauchy}:
\begin{align}
T_{\varepsilon}= \exp(\kappa/\varepsilon),
\end{align}
where $\kappa$ a positive constant independent of $\varepsilon$. So we complete the proof of Theorem \ref{mainthm}.
\section{Proof of Theorem \ref{mainthm33}}\label{sec3hhh}
In this section, we will prove Theorem \ref{mainthm33} concerning the low regularity almost global existence in the radial symmetric case. Our strategy is to combine the KSS type estimate \eqref{KSS111} and
some weight Sobolev inequalities (see Lemma \ref{mingti89}), and exploit the fact $\widetilde{\Omega} u=0$ sufficiently.
 \begin{rem}
 We note that in the radial symmetric case, the nonlinear elastic wave equation can be reduced to a system of wave equations, but as far as the proof of almost global existence is concerned,
this reduction is not necessary.
\end{rem}
\par
\begin{Lemma}\label{mingti89}
For~$u\in C_{0}^{\infty}(\mathbb{R}^{3};\mathbb{R}^{3})$, we have
\begin{align}\label{sharp34533rtg}
\|r^{1/2}u\|_{L^{\infty}(r\geq 1)}&\leq C\sum_{|a|\leq 1}\|\langle r\rangle^{-1/2}\nabla\widetilde{\Omega}^{a}u\|_{L^2(\mathbb{R}^{3})}+C\sum_{|a|\leq 2}\|\langle r\rangle^{-1/2}\widetilde{\Omega}^{a}u\|_{L^2(\mathbb{R}^{3})},\\\label{rtffsharp34533rtg}
\|r^{1/2-\delta}u\|_{L^{\infty}(r\leq 1)}&\leq C\sum_{|a|\leq 1}\|\langle r\rangle^{-\delta}r^{-1/2+\delta}\nabla\widetilde{\Omega}^{a}u\|_{L^2(\mathbb{R}^{3})}+C\sum_{|a|\leq 1}\|\langle r\rangle^{-\delta}r^{-1/2+\delta}\widetilde{\Omega}^{a}u\|_{L^2(\mathbb{R}^{3})}\nonumber\\
&+C\sum_{|a|\leq 1}\|\langle r\rangle^{-\delta}r^{-3/2+\delta}\widetilde{\Omega}^{a}u\|_{L^2(\mathbb{R}^{3})},~0<\delta\leq 1/4.
\end{align}
\end{Lemma}
\begin{proof}
\eqref{sharp34533rtg} follows from replacing $u$ by $ r^{-1/2} u$ in \eqref{hjkiile}.
For \eqref{rtffsharp34533rtg}, first we have (see (3.14a) of \cite{Sideris00})
\begin{align}\label{gjkuk}
\|r^{1/2}u\|_{L^{\infty}(\mathbb{R}^{3})}&\leq C\sum_{|a|\leq 1}\|\nabla\widetilde{\Omega}^{a}u\|_{L^2(\mathbb{R}^{3})}.
\end{align}
Noting that $0<\delta\leq 1/4$, we have
\begin{align}
\|r^{1/2-\delta}u\|_{L^{\infty}(r\leq 1)}\leq C\|r^{1/2}\langle r\rangle^{-\delta}r^{-1/2+\delta}u\|_{L^{\infty}(\mathbb{R}^{3})}.
\end{align}
We can prove \eqref{rtffsharp34533rtg} by replacing $u$ by $\langle r\rangle^{-\delta} r^{-1/2+\delta} u$ in \eqref{gjkuk}.
\end{proof}
\par
Similarly to Lemma \ref{Nulll}, we have the following
\begin{Lemma}\label{Nulll000}
For any solution $u$~of \eqref{Cauchy} in ${X}^{k}(T)$,
we have
\begin{align}
L\nabla^{a}u=\sum_{b+c=a}N(\nabla^{b}u,\nabla^{c}u),
\end{align}
in which the sum extends over ordered partitions of the sequences $a$, with $|a|\le k-1$.
\end{Lemma}
In order to prove Theorem \ref{mainthm33}, the key point is the following a priori estimate.
 \begin{proposition}\label{mainthm33333}
There exist positive constants $\kappa$, $\varepsilon_0$ and $A$ such that for any given $\varepsilon$ with $0<\varepsilon\leq \varepsilon_{0}$, if the initial data is radially symmetric and satisfies
\begin{align}\label{label345neweff}
\|\nabla u_0\|_{H^2}+\|u_1\|_{H^2}\leq \varepsilon,
\end{align}
and $u$ is a smooth and radially symmetric solution
to Cauchy problem \eqref{Cauchy}, then for any $T\leq \exp({\kappa}/{\varepsilon})$,
\begin{align}
\|u\|_{X^{3}(T)}\leq  A \varepsilon.
\end{align}
\end{proposition}
\par
Based on the method of proving Proposition \ref{mainthm33333},  by some density argument and contraction-mapping argument, we can show Theorem
 \ref{mainthm33}. Because this procedure is routine and in order to keep the paper to a moderate length, we will omit it and refer the reader to \cite{MR2919103} and \cite{MR2262091}.
\par
Now we will prove Proposition \ref{mainthm33333}. Assume that $u=u(t,x)$ is a smooth and radially symmetric solution of the Cauchy problem \eqref{Cauchy} on $[0,T]$.
 We will show that there exist positive constants $\kappa$, $\varepsilon_0$ and $A$ such that for any $T\leq \exp({\kappa}/{\varepsilon})$, we have $\|u\|_{X^{3}(T)}\leq  A \varepsilon$ under the assumption that \eqref{label345neweff} and $\|u\|_{X^{3}(T)}\leq 2A \varepsilon$, where $0<\varepsilon\leq\varepsilon_0$.\par
 By Lemma \ref{Nulll000} and \eqref{KSS111} in Theorem \ref{thmKSS00}, we can get
 \begin{align}\label{byby133}
\|u\|_{X^{3}(T)}\leq &C_1\varepsilon+C\sum_{|a|\leq 2}\|\partial \nabla u\nabla \nabla^au\|_{L^1_tL^2_x(S_T)}+C\sum_{|a|\leq 2}\| \nabla u\nabla \nabla^au\|_{L^1_tL^2_x(S_T)}\nonumber\\
 &+C(\log(2+T))^{1/2}\sum_{|a|\leq 2}\|\langle r\rangle^{-\delta}{ r}^{-1/2+\delta}{\nabla u}\nabla \nabla^{a}u\|_{L^2_{t,x}(S_{T})}\nonumber\\
 &+C\sum_{|a|\leq 2}\sum_{\substack{b+c=a\\b,c\neq a}}\|\nabla \nabla^bu\nabla^2 \nabla^cu\|_{L^1_tL^2_x(S_T)}.
 \end{align}
 For $|a|\leq 2$, we have
 \begin{align}\label{similar33}
\|\partial \nabla u\nabla \nabla^au\|_{L^1_tL^2_x(S_T)}
 \leq \|\langle r\rangle^{-\delta} r^{-1/2+\delta}\nabla \nabla^au\|_{L^2_{t,x}(S_T)}\| \langle r\rangle^{\delta} r^{1/2-\delta}\partial \nabla u\|_{L^2_{t} L^{\infty}_{x}(S_T)}.
 \end{align}
Noting that $\widetilde{\Omega}u=0$ and the commutation relationship \eqref{compoi}, it follows from \eqref{sharp34533rtg} that
\begin{align}\label{similar3355}
&\|  \langle r\rangle^{\delta} r^{1/2-\delta}\partial \nabla u\|_{L^{\infty}(r\geq 1)}\nonumber\\
&\leq C\|   r^{1/2}\partial \nabla u\|_{L^{\infty}(r\geq 1)}\leq C\sum_{|\alpha|\leq 2}\| \langle r\rangle^{-1/2}\partial \nabla^{\alpha}u\|_{L^2(\mathbb{R}^{3})}\nonumber\\
&\leq C\sum_{|\alpha|\leq 2}\|\langle r\rangle^{-\delta}r^{-1/2+\delta}\partial \nabla^{\alpha}u\|_{L^2(\mathbb{R}^{3})}.
 \end{align}
 And by \eqref{rtffsharp34533rtg} we can get
 \begin{align}\label{similar3rt3fgggh}
&\|  \langle r\rangle^{\delta} r^{1/2-\delta}\partial \nabla u\|_{L^{\infty}(r\leq 1)}\nonumber\\
&\leq C\|   r^{1/2-\delta}\partial \nabla u\|_{L^{\infty}(r\leq 1)}\nonumber\\
&\leq C\sum_{|\alpha|\leq 2}\|\langle r\rangle^{-\delta}r^{-1/2+\delta}\partial \nabla^{\alpha}u\|_{L^2(\mathbb{R}^{3})}
+C\|\langle r\rangle^{-\delta}r^{-3/2+\delta}\partial \nabla u\|_{L^2(\mathbb{R}^{3})}.
 \end{align}
 It follows from \eqref{similar33}, \eqref{similar3355} and \eqref{similar3rt3fgggh} that
  \begin{align}
\|\partial \nabla u\nabla \nabla^au\|_{L^1_tL^2_x(S_T)}
\leq C (\log(2+T))\|u\|_{X^{3}(T)}^2.
 \end{align}
 By similar argument, we can also show that for $|a|\leq 2$, it holds that
 \begin{align}
 \|\nabla u\nabla \nabla^au\|_{L^1_tL^2_x(S_T)}
 \leq C (\log(2+T))\|u\|_{X^{3}(T)}^2,
 \end{align}
and for $|a|\leq 2, b+c=a,b,c\neq a$ (i.e., $|b|\leq 1, |c|\leq 1$), we also have
 \begin{align}\label{byby4}
\|\nabla \nabla^bu\nabla^2 \nabla^cu\|_{L^1_tL^2_x(S_T)}\leq C (\log(2+T))\|u\|_{X^{3}(T)}^2.
 \end{align}
 \par
  For $|a|\leq 2$, it follows from the H\"{o}lder inequality and the Sobolev embedding $H^2(\mathbb{R}^{3})\hookrightarrow L^{\infty}(\mathbb{R}^{3})$ that
 \begin{align}
 &(\log(2+T))^{1/2}\|\langle r\rangle^{-\delta}{ r}^{-1/2+\delta}{\nabla u}\nabla \nabla^{a}u\|_{L^2_{t,x}(S_{T})}\nonumber\\
 &\leq (\log(2+T))^{1/2}\|\langle r\rangle^{-\delta}{ r}^{-1/2+\delta}\nabla \nabla^{a}u\|_{L^2_{t,x}(S_{T})}\|{\nabla u}\|_{L^{\infty}_{t,x}(S_{T})}\nonumber\\
 &\leq C(\log(2+T))^{1/2}\|\langle r\rangle^{-\delta}{ r}^{-1/2+\delta}\nabla \nabla^{a}u\|_{L^2_{t,x}(S_{T})}\sum_{|\alpha|\leq 2}\|\nabla \nabla^{\alpha}u\|_{L^{\infty}_{t}L^2_{x}(S_T)}\nonumber\\
 &\leq C (\log(2+T))\|u\|_{X^{3}(T)}^2.
 \end{align}

 \par

Hence the above argument gives
\begin{align}
\|u\|_{X^{3}(T)}\leq C_1\varepsilon+C_2(\log(2+T))\|u\|_{X^{3}(T)}^2
\leq C_1\varepsilon+4C_2(\log(2+T))A^2\varepsilon^2.
\end{align}
Take $A=4C_1$ and $\varepsilon_0>0$ sufficiently small.
Then for any $0<\varepsilon\leq\varepsilon_0$, if
\begin{align}\label{ertftt123}
16C_2(\log(2+T))A\varepsilon\leq 1,
\end{align}
then it holds that
\begin{align}
\|u\|_{X^{3}(T)}\leq A\varepsilon.
\end{align}
Consequently, it follows from \eqref{ertftt123} that we can get the lifespan estimate:
\begin{align}
T_{\varepsilon}= \exp(\kappa/\varepsilon),
\end{align}
where $\kappa$ a positive constant independent of $\varepsilon$. So we complete the proof of Proposition \ref{mainthm33333}.
\section*{Acknowledgements}
The first author was supported in part by
JSPS KAKENHI Grant Number JP15K04955.
The second author was supported by Shanghai Sailing Program (No.17YF1400700) and the Fundamental Research Funds for the Central Universities
(No.17D110913).


\begin{thebibliography}{10}

\bibitem{Agemi00}
R.~Agemi, \href{http://dx.doi.org/10.1007/s002220000084}{Global existence of
  nonlinear elastic waves}, Invent. Math. 142 (2000) 225--250.

\bibitem{MR1218879}
A.~J. Chorin, J.~E. Marsden,
  \href{http://dx.doi.org/10.1007/978-1-4612-0883-9}{\emph{A mathematical
  introduction to fluid mechanics}}, \emph{Texts in Applied Mathematics},
  vol.~4, Springer-Verlag, New York, 3rd ed., 1993.

\bibitem{MR3286047}
Z.~Guo, Y.~Wang, \href{https://doi.org/10.1007/s11854-014-0025-6}{Improved
  {S}trichartz estimates for a class of dispersive equations in the radial case
  and their applications to nonlinear {S}chr\"odinger and wave equations}, J.
  Anal. Math. 124 (2014) 1--38.

\bibitem{MR2569550}
K.~Hidano, \href{https://doi.org/10.4171/RMI/579}{Small solutions to
  semi-linear wave equations with radial data of critical regularity}, Rev.
  Mat. Iberoam. 25 (2009) 693--708.

\bibitem{Hidano}
K.~Hidano, \href{https://arxiv.org/abs/1610.04824}{Regularity and lifespan of
  small solutions to systems of quasi-linear wave equations with multiple
  speeds, {I}: almost global existence}, arXiv:1610.04824v2. Accepted for
  publication in RIMS K\^{o}ky\^{u}roku Bessatsu. (2016).

\bibitem{Hidano17}
K.~Hidano, J.~Jiang, S.~Lee, C.~Wang,
  \href{https://arxiv.org/abs/1605.06748v3}{Weighted fractional chain rule and
  nonlinear wave equations with minimal regularity}, arXiv:1605.06748v3
  (2017).

\bibitem{MR2919103}
K.~Hidano, C.~Wang, K.~Yokoyama,
  \href{http://projecteuclid.org/euclid.ade/1355703087}{On almost global
  existence and local well posedness for some 3-{D} quasi-linear wave
  equations}, Adv. Differential Equations 17 (2012) 267--306.

\bibitem{MR3552253}
K.~Hidano, C.~Wang, K.~Yokoyama,
  \href{http://dx.doi.org/10.1007/s00208-015-1346-1}{Combined effects of two
  nonlinearities in lifespan of small solutions to semi-linear wave equations},
  Math. Ann. 366 (2016) 667--694.

\bibitem{MR2262091}
K.~Hidano, K.~Yokoyama,
  \href{http://projecteuclid.org/euclid.die/1356050327}{Space-time
  {$L^2$}-estimates and life span of the {K}lainerman-{M}achedon radial
  solutions to some semi-linear wave equations}, Differential Integral
  Equations 19 (2006) 961--980.

\bibitem{MR1482992}
T.~Hoshiro, \href{http://dx.doi.org/10.1007/BF02843156}{On weighted {$L^2$}
  estimates of solutions to wave equations}, J. Anal. Math. 72 (1997) 127--140.

\bibitem{MR0420024}
T.~J.~R. Hughes, T.~Kato, J.~E. Marsden,
  \href{http://dx.doi.org/10.1007/BF00251584}{Well-posed quasi-linear
  second-order hyperbolic systems with applications to nonlinear elastodynamics
  and general relativity}, Arch. Rational Mech. Anal. 63 (1976) 273--294.

\bibitem{MR2911108}
J.-C. Jiang, C.~Wang, X.~Yu,
  \href{https://doi.org/10.3934/cpaa.2012.11.1723}{Generalized and weighted
  {S}trichartz estimates}, Commun. Pure Appl. Anal. 11 (2012) 1723--1752.

\bibitem{MR1774100}
S.~Jiang, R.~Racke, \emph{Evolution equations in thermoelasticity},
  \emph{Chapman \& Hall/CRC Monographs and Surveys in Pure and Applied
  Mathematics}, vol. 112, Chapman \& Hall/CRC, Boca Raton, FL, 2000.

\bibitem{John84}
F.~John, \href{http://dx.doi.org/10.1007/3-540-12916-2_58}{Formation of
  singularities in elastic waves}, in \emph{Trends and applications of pure
  mathematics to mechanics ({P}alaiseau, 1983)}, \emph{Lecture Notes in Phys.},
  vol. 195, Springer, Berlin, 1984, 194--210.

\bibitem{John88}
F.~John, \href{http://dx.doi.org/10.1002/cpa.3160410507}{Almost global
  existence of elastic waves of finite amplitude arising from small initial
  disturbances}, Comm. Pure Appl. Math. 41 (1988) 615--666.

\bibitem{MR1945285}
M.~Keel, H.~F. Smith, C.~D. Sogge,
  \href{http://dx.doi.org/10.1007/BF02868477}{Almost global existence for some
  semilinear wave equations}, J. Anal. Math. 87 (2002) 265--279.

\bibitem{Klainerman98}
S.~Klainerman,
  \href{http://dx.doi.org/10.1002/(SICI)1097-0312(199809/10)51:9/10<991::AID-CPA3>3.3.CO;2-0}{On
  the work and legacy of {F}ritz {J}ohn, 1934--1991}, Comm. Pure Appl. Math. 51
  (1998) 991--1017.

\bibitem{MR1231427}
S.~Klainerman, M.~Machedon,
  \href{http://dx.doi.org/10.1002/cpa.3160460902}{Space-time estimates for null
  forms and the local existence theorem}, Comm. Pure Appl. Math. 46 (1993)
  1221--1268.

\bibitem{Klainerman96}
S.~Klainerman, T.~C. Sideris,
  \href{http://dx.doi.org/10.1002/(SICI)1097-0312(199603)49:3<307::AID-CPA4>3.0.CO;2-H}{On
  almost global existence for nonrelativistic wave equations in {$3$}{D}},
  Comm. Pure Appl. Math. 49 (1996) 307--321.

\bibitem{MR3014806}
H.~Kubo, \href{http://dx.doi.org/10.1007/978-3-0348-0454-7_10}{Lower bounds for
  the lifespan of solutions to nonlinear wave equations in elasticity}, in
  \emph{Evolution equations of hyperbolic and {S}chr\"odinger type},
  \emph{Progr. Math.}, vol. 301, Birkh\"auser/Springer Basel AG, Basel, 2012,
  187--212.

\bibitem{MR1375301}
H.~Lindblad,
  \href{http://muse.jhu.edu/journals/american_journal_of_mathematics/v118/118.1lindblad.pdf}{Counterexamples
  to local existence for semi-linear wave equations}, Amer. J. Math. 118 (1996)
  1--16.

\bibitem{MR1666844}
H.~Lindblad, \href{http://dx.doi.org/10.4310/MRL.1998.v5.n5.a5}{Counterexamples
  to local existence for quasilinear wave equations}, Math. Res. Lett. 5 (1998)
  605--622.

\bibitem{MR1335386}
H.~Lindblad, C.~D. Sogge, \href{https://doi.org/10.1006/jfan.1995.1075}{On
  existence and scattering with minimal regularity for semilinear wave
  equations}, J. Funct. Anal. 130 (1995) 357--426.

\bibitem{Liu}
M.~Liu, C.~Wang, \href{https://arxiv.org/abs/1709.00967}{Global existence for
  some 4-{D} quasilinear wave equations with low regularity}, arXiv:1709.00967
  (2017).

\bibitem{MR2108356}
S.~Machihara, M.~Nakamura, K.~Nakanishi, T.~Ozawa,
  \href{https://doi.org/10.1016/j.jfa.2004.07.005}{Endpoint {S}trichartz
  estimates and global solutions for the nonlinear {D}irac equation}, J. Funct.
  Anal. 219 (2005) 1--20.

\bibitem{MR2249996}
J.~Metcalfe, \href{http://dx.doi.org/10.1155/IMRN/2006/69826}{Elastic waves in
  exterior domains. {I}. {A}lmost global existence}, Int. Math. Res. Not.
  (2006) Art. ID 69826, 41.

\bibitem{MR2217314}
J.~Metcalfe, C.~D. Sogge, \href{http://dx.doi.org/10.1137/050627149}{Long-time
  existence of quasilinear wave equations exterior to star-shaped obstacles via
  energy methods}, SIAM J. Math. Anal. 38 (2006) 188--209.

\bibitem{MR2299569}
J.~Metcalfe, C.~D. Sogge,
  \href{http://dx.doi.org/10.1007/s00209-006-0083-2}{Global existence of
  null-form wave equations in exterior domains}, Math. Z. 256 (2007) 521--549.

\bibitem{MR2344575}
J.~Metcalfe, B.~Thomases, \href{http://dx.doi.org/10.1093/imrn/rnm034}{Elastic
  waves in exterior domains. {II}. {G}lobal existence with a null structure},
  Int. Math. Res. Not. IMRN  (2007) Art. ID rnm034, 43.

\bibitem{MR2971205}
E.~Y. Ovcharov, \href{https://doi.org/10.1080/03605302.2011.632047}{Radial
  {S}trichartz estimates with application to the 2-{D} {D}irac-{K}lein-{G}ordon
  system}, Comm. Partial Differential Equations 37 (2012) 1754--1788.

\bibitem{MR1211729}
G.~Ponce, T.~C. Sideris, \href{https://doi.org/10.1080/03605309308820925}{Local
  regularity of nonlinear wave equations in three space dimensions}, Comm.
  Partial Differential Equations 18 (1993) 169--177.

\bibitem{Sideris96}
T.~C. Sideris, \href{http://dx.doi.org/10.1007/s002220050030}{The null
  condition and global existence of nonlinear elastic waves}, Invent. Math. 123
  (1996) 323--342.

\bibitem{Sideris00}
T.~C. Sideris, \href{http://dx.doi.org/10.2307/121050}{Nonresonance and global
  existence of prestressed nonlinear elastic waves}, Ann. of Math. (2) 151
  (2000) 849--874.

\bibitem{Sideris01}
T.~C. Sideris, S.-Y. Tu,
  \href{http://dx.doi.org/10.1137/S0036141000378966}{Global existence for
  systems of nonlinear wave equations in 3{D} with multiple speeds}, SIAM J.
  Math. Anal. 33 (2001) 477--488 (electronic).

\bibitem{MR2178963}
H.~F. Smith, D.~Tataru,
  \href{http://dx.doi.org/10.4007/annals.2005.162.291}{Sharp local
  well-posedness results for the nonlinear wave equation}, Ann. of Math. (2)
  162 (2005) 291--366.

\bibitem{MR2455195}
C.~D. Sogge, \emph{Lectures on non-linear wave equations}, International Press,
  Boston, MA, 2nd ed., 2008.

\bibitem{MR1232192}
E.~M. Stein, \emph{Harmonic analysis: real-variable methods, orthogonality, and
  oscillatory integrals}, \emph{Princeton Mathematical Series}, vol.~43,
  Princeton University Press, Princeton, NJ, 1993. With the assistance of
  Timothy S. Murphy, Monographs in Harmonic Analysis, III.

\bibitem{MR2128434}
J.~Sterbenz, \href{http://dx.doi.org/10.1155/IMRN.2005.187}{Angular regularity
  and {S}trichartz estimates for the wave equation, with an appendix by {I}gor
  {R}odnianski}, Int. Math. Res. Not.  (2005) 187--231.

\bibitem{MR3656947}
Q.~Wang, \href{https://doi.org/10.1007/s40818-016-0013-5}{A geometric approach
  for sharp local well-posedness of quasilinear wave equations}, Ann. PDE 3
  (2017) Art. 12, 108.

\bibitem{MR3650329}
D.~Zha, \href{http://dx.doi.org/10.1016/j.jde.2017.04.014}{Space--time {$L^2$}
  estimates for elastic waves and applications}, J. Differential Equations 263
  (2017) 1947--1965.

\bibitem{MR2482537}
Y.~Zhou, Z.~Lei, \href{https://projecteuclid.org/euclid.ade/1355867360}{Global
  low regularity solutions of quasi-linear wave equations}, Adv. Differential
  Equations 13 (2008) 55--104.

\end{thebibliography}
\end{document}